\documentclass[11pt, reqno]{amsart}
\usepackage{amsmath, amssymb, amsthm,mathrsfs,graphicx,color,fullpage,tikz}
\usepackage[all]{xypic}

\newtheorem{theorem}{Theorem}[section]
\newtheorem{coro}[theorem]{Corollary}
\newtheorem{defi}[theorem]{Definition}
\newtheorem{lemma}[theorem]{Lemma}
\newtheorem{prop}[theorem]{Proposition}

\newtheorem{rem}[theorem]{Remark}

\def\slfrac#1#2{\hbox{\kern.1em %
 \raise.5ex\hbox{\the\scriptfont0 #1}\kern-.11em %
 /\kern-.15em\lower.25ex\hbox{\the\scriptfont0 #2}}}

\newcommand{\eqn}[1]{(\ref{#1})}

\newcommand{\eeq}{\end{equation}}
\newcommand{\beql}[1]{\begin{equation}\label{#1}}

\newcommand{\ggG}{{\gamma}}

\newcommand{\pt}{\partial}

\newcommand{\ep}{\epsilon}

\newcommand{\CC}{{\mathbb C}}

\newcommand{\RR}{{\mathbb R}}
\newcommand{\ZZ}{{\mathbb Z}}

\newcommand{\DLp}{{D_L^{+}}}
\newcommand{\DLm} {{D_L^{-}}}

\newcommand{\sA}{{\mathcal A}}
\newcommand{\sB}{{\mathcal B}}
\newcommand{\sD}{{\mathcal D}}
\newcommand{\sE}{{\mathcal E}}

\newcommand{\sH}{{\mathcal H}}

\newcommand{\sK}{{\mathcal K}}

\newcommand{\sP}{{\mathcal P}}

\newcommand{\sS}{{\mathcal S}}

\newcommand{\hD}{{\rm D}}
\newcommand{\hI}{{\rm I}}
\newcommand{\hJ}{{\rm J}}

\newcommand{\hR}{{\rm R}}
\newcommand{\hS}{{\rm S}}
\newcommand{\hT}{{\rm T}}
\newcommand{\hZ}{{\rm Z}}

\newcommand{\rC}{{\rm C}}

\newcommand{\sgn} {{\rm sgn}}
\newcommand{\Bx}{{\Box}}
\newcommand{\Bxc}{{\Box^{\circ}}}

\title{The Lerch Zeta Function  IV.  Hecke Operators}

\author{Jeffrey C. Lagarias}
\thanks{The research of the first author was supported by NSF grants
DMS-0801029, DMS-1101373 and DMS-1401224,  that of the second author by NSF grant DMS-1101368  
and Simons Foundation grant \#355798.}
\address{Department of Mathematics, University of Michigan,
Ann Arbor, MI 48109-1043,USA}
\email{lagarias@umich.edu}
\author{Wen-Ching Winnie Li}
\address{Department of Mathematics, Pennsylvania State University,
University Park, PA 16802-8401, USA}

\email{ wli@math.psu.edu}
\date{August 11, 2016}
\begin{document}
\begin{abstract} 
This paper studies  algebraic and analytic structures
associated with the Lerch zeta function.
It defines  a family of 
two-variable Hecke operators $\{ \hT_m: \, m \ge 1\}$ given by 
$\hT_m(f)(a, c) = \frac{1}{m} \sum_{k=0}^{m-1} f(\frac{a+k}{m}, mc)$
acting on certain spaces of real-analytic functions, including Lerch zeta
functions for various parameter values. 
The actions of various related operators on these function spaces are determined.
It is shown that, for each $s \in \CC$, there is 
a two-dimensional vector space spanned
by linear combinations of  Lerch zeta functions characterized as 
a maximal space of simultaneous eigenfunctions
for this family of Hecke operators.  This is an
analogue of  a result of Milnor for
the Hurwitz zeta function. We also relate these functions to a linear partial
differential operator in the $(a, c)$-variables having the Lerch zeta function 
as an eigenfunction.
\end{abstract}

\maketitle
\tableofcontents

\noindent

\setlength{\baselineskip}{1.2\baselineskip}
%
%
%
\section{Introduction}
\setcounter{equation}{0}

The {\em Lerch zeta function}
is  defined by the series
\beql{101}
\zeta (s,a,c) = \sum_{n=0}^\infty \frac{e^{2\pi ina}}{ (n+c)^{s}} 
\eeq
which absolutely converges for complex variables $\Re(s)> 1$, $\Re(c)>0$ and $\Im(a) \ge 0$.
In this paper we will restrict attention to $(a, c)$ being real variables. In
that case, excluding integer values of $a$ and $c$,  it conditionally converges
to an analytic function for $\Re(s)>0$ and analytically continues in the $s$-variable
to the entire plane. When either $a$ or $c$ is an integer (or both) it analytically
continues in $s$ to a meromorphic function whose singularity set 
is  at most a single simple pole, located
at $s=1$, { see \cite[Theorem 2.2]{LL1}.}

Two key  properties of the Lerch zeta function $\zeta(s, a ,c)$ for real $(a, c)$,
restricting to $0< a, c <1$, are the following.
\begin{enumerate}
\item[(1)]
It is an eigenfunction of a linear partial differential operator 
\begin{eqnarray}\label{DL}
D_L := \frac{1}{2 \pi i} \frac{\partial}{\partial a} \frac{\partial}{\partial c} + c \frac{\partial}{\partial c}.
\end{eqnarray}
It satisfies 
{
\begin{equation}\label{Eigenvalue}
(D_L \zeta)(s, a, c) = -s \zeta(s, a,c).
\end{equation}
}
This property was noted in Parts II and III (\cite{LL2}, \cite{LL3}).
When restricting to real variables, we regard $\frac{\partial}{\partial a}$ and $\frac{\partial}{\partial c}$ as
real differential operators. (They were treated as complex differential operators in \cite{LL2} and \cite{LL3}.)
\item[(2)]
{ In particular, it} satisfies two four-term  functional equations encoding a discrete symmetry under
$(s, a, c)$ to $(1-s, 1-c, a)$. 
These
functional equations
were  noted by
Weil \cite{Wei76}
under the restriction $0<a <1$ and $0<c<1$, and were studied in Part I (\cite{LL1}).
{ To state the functional equations, let} 
$$L^{\pm}(s, a, c) :=\zeta(s, a, c) \pm e^{-2 \pi i a} \zeta(s, 1-a, 1-c)$$
 and
define the completed functions
$$
\hat{L}^{\pm} (s, a , c) :=\pi^{- \frac{s+ \ep}{2} } \Gamma( \frac{s+\ep}{2} ) L^{\pm}(s, a, c),
$$
where $\ep := \frac{1}{2}(1 - (\pm1))$ takes values $0$ or $1$. Then the functional equations are
\begin{equation}\label{fun1}
\hat{L}^{+}(s, a, c) = e^{-2 \pi i a c} \hat{L}^{+}(1-s, 1-c , a)
\end{equation}
and
\begin{equation}\label{fun2}
\hat{L}^{-}(s, a,c) = i\,e^{-2 \pi i ac} \hat{L}^{-}(1-s, 1-c , a).
\end{equation}
\end{enumerate}

In this paper we will  study a  family $\{\hT_m: m \ge 1\}$ of  ``Hecke operators'' 
which also preserve the Lerch zeta function, in the sense that it 
is a simultaneous eigenfunction of these operators.
The {\em (two-variable)  Hecke operators} $\hT_m$ are formally defined by
\beql{102}
\hT_m (f)(a,c) := \frac{1}{m} \sum_{k=0}^{m-1} f \left( \frac{a+k}{m}, mc \right).
\eeq
The operators $\hT_m$ stretch coordinates in the $c$-direction, while each
individual term on the right side of \eqn{102}
contracts coordinates in the $a$-direction, and they (formally) commute.
Two-variable 
operators of  this  form   were  originally
  introduced  in 1989 by Zhi Wei Sun \cite{Su89} in a completely different context,
that of covering systems of arithmetic progressions of integers.
In 2000 S. Porubsky \cite[Sect. 4]{Po00} noted that the Lerch zeta function
is an eigenfunction of these operators on a suitable domain.
For  $\Re(s) >1$ the function
 $\zeta(s, a, c)$ is well defined and real-analytic in 
the variables $(a, c)$ in  the domain 
$$\sD_{++} := \{ (a, c) \in \RR_{>0} \times \RR_{>0}\},$$
which is the first quadrant in the $(a, c)$-plane. 
 One can now check that   for $\Re(s) >1$
  the Lerch zeta function
$\zeta(s, a,c)$ has  on this domain the simultaneous  eigenfunction property  
\beql{103aa}
\hT_m (\zeta)(s, a, c) = m^{-s} \zeta(s, a, c)
\eeq
holding  for all $m \ge 1$.

The question this paper considers is that of obtaining an  
extension of the operators  above {inside suitable function spaces, in which 
the Lerch zeta function will be a simultaneous eigenfunction for all $s \in \CC$.}
 We will show there is such
an extension, and study the action of all these operators on
the ensuing function spaces.  We then obtain a characterization of 
simultaneous eigenfunction solutions of these operators 
in these function spaces,  in the spirit of Milnor \cite{Mi83}.


\subsection{{Hecke operators and function spaces}}\label{sec11}

The definition of the  two-variable Hecke operators incorporates  a dilation in one variable
that seems to require it to be defined on an unbounded domain.  As introduced it  makes sense
most naturally on a domain like $[0,1] \times \RR$. 
Obstacles to including the Lerch zeta function in such a space for all ranges of $s$  include:
\begin{enumerate}
\item[(1)]
analytic continuation will be  required in the $s$-variable to cover all $s \in \CC$; 
\item[(2)] the functional
equations do not  leave the domain $\sD_{++}$ invariant; 
\item[(3)] the analytically continued
Lerch zeta function has  discontinuities at  parameter values  $(a, c)$  where 
$a$ is an integer or $c$ is a nonpositive integer. 
\end{enumerate}
The function spaces we consider 
 must therefore incorporate some discontinuous functions.

Our construction  builds on the real-variable results obtained in Part I \cite{LL1}, which   
gave an extension of the Lerch zeta function to real variables $(a, c)$ in $\RR \times \RR$.
The first step was to establish properties for 
 values of $(a,c)$
in the closed unit square 
$$
\Box:= \{ (a, c) : 0 \le a \le 1,\, 0 \leq c \leq 1\},
$$
 and then followed by  
introducing  a real-analytic   extension $\zeta_{\ast}(s, a,c )$ of
$\zeta(s, a, c)$ to $(a, c) \in \RR \times \RR$ 
for $\Re(s) > 1$, called the {\em extended Lerch zeta function}, defined by 
\beql{101c}
\zeta_{\ast}(s, a, c) = \sum_{n+c > 0} e^{2\pi in a} |n+c|^{-s}
\eeq
for $\Re(s) >1$. 
(The resulting function is real-analytic in the variables $a$
and $c$ separately, away from integer values of $a$ and $c$.)
The extended Lerch zeta function specializes
\footnote{There is no need to take the absolute value in this formula, but it becomes
useful in further generalizations.}
 for $0 < c<1$ to 
$$
\zeta_{\RR}(s, a ,c) := \sum_{n=0}^\infty \frac{e^{2\pi ina}}{ |n+c|^{s}}.
$$

\noindent This  extended function analytically continues in the $s$-variable
to the $s$-plane, with a suitable functional equation.  We note three features
of this extension:
\begin{enumerate}
\item
The real-analytic extension obeys {\em twisted-periodicity conditions} in
the $a$ and $c$ variables: 
\begin{eqnarray*}
\zeta_{\ast}(s, a+1, c) & = & \quad\quad\,\,\, \zeta_{\ast}(s, a, c),\\
\zeta_{\ast}(s, a, c+1)& = & e^{-2\pi ia}\zeta_{\ast}(s, a, c),
\end{eqnarray*}
\item
The extended  function satisfies two symmetrized four-term functional equations.
This fact for $0<a<1, 0 < c<1$ was 
 originally noted by  A. Weil \cite[pp. 54--58]{Wei76}. These
functional equations generalize that of the Riemann zeta
function, and  are given in Section \ref{sec51}  below.
\item
The function $\zeta_{\ast}(s, a,c)$ 
has discontinuities in the $a$ and $c$ variables at integer
values of  $a, c$, for various ranges of $s$. These discontinuities
encode information about the nature of the singularities of these functions
in the $(a, c)$-variables at integer values. 
\end{enumerate} 
The extended Lerch function $\zeta_{\ast}(s, a, c)$ will be our extension of
the Lerch function to $\RR \times \RR$; it is well-defined off of grid of horizontal
and vertical lines.

The function spaces on $\RR \times \RR$ that we consider will 
be suitable classes of 
(almost everywhere defined) functions $F(a, c)$ that satisfy the twisted-periodicity conditions above:
\begin{eqnarray*}
F(a+1,\,\, c) &= &\quad\quad \,\,\, F(a, c),\\
F(a, \,\,c+1) & =& e^{-2 \pi i a} F(a, c).
\end{eqnarray*}
The values of such functions are completely determined
by their values in the unit square $\Bx$ in the $(a, c)$-variables. 
A key feature  is that the two-variable Hecke operators preserve the twisted-periodicity  property.
{ In  consequence it suffices to study these functions inside the unit square,
since  twisted-periodicity  then  uniquely extends the function to $\RR \times \RR$, 
providing a means of defining   the 
two-variable Hecke operator action on $\RR \times \RR$ using only 
 function values  defined inside the unit square.}

Another important ingredient of our extension
is an operator encoding the  functional equations. We can reframe 
the functional equation  in terms
of a  new operator acting on functions on $\RR \times \RR$, the $\hR$-operator
defined by 
\begin{eqnarray}\label{R}
\hR (F)(a, c) := e^{-2 \pi i a c} F(1-c, a).
\end{eqnarray}
This operator { has order $4$ and it} takes the space of continuous functions defined on
the unit square $\Bx$ into itself, and 
extends to define an isometry of the function space $L^p(\Bx, da \, dc)$
for $1 \le p \le \infty$.
Furthermore when viewed as acting on  functions with domain $\RR \times \RR$,
it preserves the subspace of twisted-periodic functions.

Our results incorporating the functional equation will be invariant with respect to the action of the 
$\hR$-operator acting on suitable function spaces. 
The action of this operator leads to four families of two-variable Hecke operators,
that  take the following form.
These families are 
$$
\hT_m,~~\hS_m  :=  \hR \hT_m \hR^{-1} ~,~{\hT}_m^{\vee}  :=  \hR^2 \hT_m \hR^{-2},~~ 
\mbox{and}~~{\hS}_m^{\vee}  :=  \hR^3 \hT_m \hR^{-3},
$$
explicitly given as
\begin{eqnarray*}
{\hT}_m f(a, c) & =& \frac{1}{m} \sum_{k=0}^{m-1}f( \frac{a+k}{m}, mc), \\
{\hS}_m f (a,c) & = & \frac{1}{m} \sum_{k=0}^{m-1} 
e^{2 \pi i ka} f \left( ma, \frac{c+k}{m} \right) , \\
{\hT}_m^{\vee} f (a,c) & = &
\frac{1}{m} \sum_{k=0}^{m-1} 
e^{2\pi i (\frac{(1-m)a+k}{m})}f \left( \frac{a+k}{m} , 1 + m(c - 1) \right) , \\
{\hS}_m^{\vee} f (a,c) & = & \frac{1}{m} \sum_{k=0}^{m-1} 
e^{2 \pi i (m-(k+1)) a}f \left( 1 + m(a - 1), \frac{c+m-(k+1)}{m}  \right) .
\end{eqnarray*}

Each of these families of operators leaves invariant one side of the unit square in $(a, c)$-coordinates,
and between them they account for all four sides of the unit square. 
The members of each  family  commute with other members
of the same family. A more  surprising result is  that, in the function spaces we consider below, 
 members of different families of these operators also commute.

To obtain results that apply to the four families {and} for  all complex $s$, we introduce
 function spaces of piecewise continuous functions, 
 allowing discontinuities. The restriction to piecewise continuous
 functions, without completing the function space, is made because  the singularities of the Lerch zeta function at
the boundary of the unit square become large as $\Re(s) \to -\infty$. 
For values of $s$ inside the critical strip  $0 < \Re(s)<1$ we are able to work 
inside various Banach  spaces.
Of particular  interest is the Banach  space $L^1( \Box, da\, dc)$,
which is relevant since the functions $L^{\pm}(s, a, c)$ belong to this space inside
the critical strip (\cite[Theorem 2.4]{LL1}).
More generally one may consider the  spaces $L^p(\Box, da \, dc)$ for $1 \le p \le 2$.
Note that for  twisted-periodic functions $F(a, c)$
the  $L^p$-norm of  $F(a,c)$ is invariant
under measurement in any unit square $x_0 \le a \le x_0+1, y_0 \le c \le y_0 +1$,
allowing real values $x_0, y_0$.

The Hilbert space $\sH= L^2(\Bx, da \, dc)$ is also of great interest; however the Lerch
zeta function $\zeta(s, a, c)$ for $(a, c) \in \Bx$ does not belong to  this space for any $s \in \CC$; for $\Re(s) >1$
this follows from the single term $c^{-s}$ not being in $\sH$; while for other $s$ it is more
subtle, see the proof of  \cite[Theorem 2.4]{LL1}. The operators $\hR$ and the four families of Hecke operators
 induce well-defined actions
of bounded operators on this space and together generate an interesting
noncommutative  algebra of operators that
is a $\star$-algebra in $\sB(\sH)$ which seems worthy of further study.

To obtain  function spaces that allow an  action of differential operators, we  restrict to suitable smaller
spaces of piecewise smooth
functions that allow discontinuities located at a fixed lattice of axis- parallel vertical lines
and horizontal lines, with relevant coordinate vector (horizontal and/or vertical) contained  in a lattice $\frac{1}{d}\ZZ$ for a fixed
$d$ depending on the function. We  can start this construction with such functions
defined on the unit square with specified 
discontinuities produced by twisted-periodicity having relative coordinate vectors in the same lattice $\frac{1}{d} \ZZ$.


\subsection{{Motivation and results}}\label{sec12}

One motivation for this study concerns finding an operator extension  of the Riemann zeta function. 
We may  formally define a {\em zeta operator} $\hZ$ acting on the domain $\sD_{++}$ as 
 $$
 \hZ : = \sum_{m=1}^{\infty}\, {\hT}_m.
 $$
 Now one can check that for 
  a fixed complex value $s$ with $\Re(s) >1$  this operator  has the Lerch zeta function $\zeta(s, a, c)$ as an eigenfunction and 
  the Riemann zeta value  $\zeta(s)$ as an eigenvalue; that is,
 $$
 \hZ(\zeta)(s, a, c) := \zeta(s) \zeta(s, a, c).
 $$
 We do not know how to  make direct sense  of the operator $\hZ$ in a larger domain
 in the $s$-parameter, but the problem to do so 
  provides a motivation for extending the action of the individual two-variable Hecke operators $\hT_m$ 
  to suitable function spaces, which this paper addresses. {In the process,   we  obtain an extension
   inside $L^1( \Box, da\, dc)$  to the parameter range  $0 < \Re(s) <1$ in which 
  $\hT_m(\zeta(s, a, c))= m^{-s} \zeta(s, a, c)$ holds 
  for all individual Hecke operators $\hT_m$.}

Precise statements of the main results are given in Section \ref{sec2},
but we first make  some general remarks.
 
\begin{enumerate}
\item[(1)] We introduce a set of auxiliary operators acting on 
twisted-periodic  function spaces, a
 linear partial differential operator $\hD_L$,
a unitary  operator $\hR$,  and the family of two-variable Hecke operators $\hT_m$,
viewed as acting on the unit square $\Bx$
through the use of twisted-periodic function spaces.
 We determine commutation relations among all these operators
on these function spaces.  

\item[(2)]
 For each $s \in \CC$ we define a two-dimensional vector space $\sE_s$,
the {\em Lerch eigenspace}, consisting of
twisted-periodic real-analytic functions defined on  $\RR^2 \smallsetminus \ZZ^2$ 
(but sometimes discontinuous on the grid $(\ZZ \times \RR) \cup (\RR \times \ZZ)$)
which satisfy 
the eigenvalue identity
\beql{104a}
\hT_m (\zeta_{\ast})(s, a, c) = m^{-s} \zeta_{\ast} (s, a, c)
\eeq
simultaneously for all $m \ge 1$. The spaces $\sE_s$ are preserved
by the $\hR$-operator, and those at $s$ and $1-s$ are related
under the symmetries of the functional equation.

\item[(3)]
 We  give a simultaneous Hecke eigenfunction interpretation of  the Lerch zeta function,
in the spirit of Milnor's \cite{Mi83} characterization of Kubert functions, including
the Hurwitz zeta function. In Section \ref{sec6} we show that for each $s \in \CC$
there is a  two-dimensional vector space $\sE_s$ of simultaneous
eigenfunctions, the {Lerch eigenspace}, sartisfying suitable 
integrability side conditions.
This is a generalization of Milnor's converse
result characterizing the Hurwitz zeta function and Kubert functions.
\end{enumerate}

In  the concluding section
we discuss the 
possible  relevance to the Riemann hypothesis of the
structures associated to the Lerch zeta function 
studied in this series of papers.
 Many previous formulations of
 the Riemann hypothesis have analogues in terms of 
 properties of the Lerch zeta function.
 The operator $D_L$ has features suggested for
 a putative ``Hilbert-Polya'' operator.  \smallskip

%
%
%
\subsection{Related work}\label{sec13}

We have already mentioned the 
definition of  two-variable Hecke
 operators  in 1989 by Zhi Wei Sun \cite{Su89} in connection with covering systems, 
as well as the work of  Porubsky \cite{Po00} noting
the eigenfunction property \eqref{103aa} of the Lerch zeta function.
See also   a survey paper of 
 Porubsky and Sch\"{o}nheim \cite{PS02}
on Erd\H{o}s's work on covering systems.

Related operators in one variable appeared earlier  in 
a number of different contexts.
{ On specializing to  functions $f(a, c)$ that are constant in the
$c$-variable, we obtain} 
one-variable difference operators  of  the form  
\beql{111}
\hT_m(f)(a) := \frac{1}{m} \sum_{k=0}^{m-1} f \left( \frac{a+k}{m}\right).
\eeq
This operator coincides for $m=q$ with an  operator  $U_q^{\ast}$ introduced 
in 1970 by 
Atkin and Lehner \cite[p. 138]{AL70}, which they termed a {\em Hecke operator.}
They motivate their choice of name ``Hecke operator'' to 1957 work of Wohlfahrt \cite{Wol57},
and we follow their terminology in naming the  two-variable operators \eqref{102}.
{  These operators  also appear,  denoted  $U_m$, in the modular forms  literature 
 as an action on $q$-expansions, 
 as $U_m( \sum_{n=0}^{\infty} a_n q^n) := \sum_{n=0}^{\infty} a_{mn} q^n$.
 If now $f= \sum_{n=0}^{\infty} a_n q^n$, then under  the change of variable $q= e^{2 \pi i \tau}$ 
 with $\tau$ in the upper half-plance, this action becomes
 $U_m(f) (\tau) = \frac{1}{m} \sum_{k=0}^{m-1} f(\frac{\tau +k}{m}),$
  see Lang \cite[p. 108 ]{Lang76}, Koblitz \cite[pp. 161-2]{Kob84}.
 In an analogous  $p$-adic modular form context the operator $U_p$
 is   sometimes called  {\em Atkin's operator}, see Dwork \cite{Dwo73} and Katz \cite{Kat73}.}
In 1983 Milnor \cite{Mi83} considered operators of form \eqref{111}, calling
them  {\em Kubert operators.}
General operators
of the type \eqn{111} acting on an abelian  group (for example $\RR/ \ZZ$) 
had previously been studied by Kubert and Lang \cite{KL76} and Kubert \cite{Ku79}.
Milnor  characterized simultaneous eigenfunction solutions
of such operators, in the space of continuous functions on the open interval $(0,1)$.
In Section \ref{sec61} we  review Milnor's  work.

{ On specializing to  functions $f(a, c)$ that are constant in the
$a$-variable, we obtain 
 a different family of one-variable operators,} the
{\em dilation operators}
\beql{110a}
\tilde{\hT}_m(f) (c) := f(mc).
\eeq
These operators satisfy the identities
$$
\tilde{\hT}_m \circ \tilde{\hT}_n= \tilde{\hT}_n \circ \tilde{\hT}_m = \tilde{\hT}_{mn}.
$$
for all $m,n \ge 1$.
The dilation operator $\tilde{\hT}_m$ coincides in form with an
 operator $A_m$ introduced in 1970
in Atkin and Lehner \cite[equation (2.1)]{AL70}\footnote{ The operator  $A_n$
appears for prime $p$ as $A_p= T_p^{\ast} - U_p^{\ast}$. However
Atkin and Lehner apply  $T_p^{\ast}$ only for $p \nmid N$, the level,
and $U_p^{\ast}$ for $p | N$, but the definitions of each operator
make sense for all $p$, so we may relate  them in the identity above.}.
{  The dilation operators  also  appear, denoted $V_m$, in the modular forms literature, 
acting on $q$-expansions as $V_m( \sum_{n=0}^{\infty} a_n q^n) = \sum_{n=0}^{\infty} a_n q^{mn}.$
If now $f= \sum_{n=0}^{\infty} a_n q^n$, then
 under  the change of variable $q= e^{2 \pi i \tau}$ 
 this action becomes $V_m(f) (\tau) = f(m \tau)$, see Lang \cite[p. 108]{Lang76},  Koblitz \cite[pp. 161-162]{Kob84}.}
 { In 1945 Beurling \cite{Beu45} studied completeness properties of dilates $\{\tilde{\hT}_m f: m \ge 1\}$ viewed in $L^2(0,1)$ of
 a periodic function $f$ on $\RR$  of period $1$.}
In 1999 B\'{a}ez-Duarte \cite{BD99} noted that this  family of dilation operators relates to
  the real-variables approach to the Riemann hypothesis due to 
  Nyman  \cite{Nym50} and Beurling \cite{Beu55},  see also B\'{a}ez-Duarte \cite{BD99} 
 and Burnol \cite{Bur04}, \cite{Bur04a}.  Convergence of series composed out of dilations of functions have
 been much studied; they include Fourier series as a special case.
 See Gaposkin \cite{Gap68}, Aisletiner et al \cite{ABS15}, \cite{ABSW15}, Berkes and Weber \cite{BW09}  and Weber \cite{Web11}.

%
%
%
\section{{Main Results}}\label{sec2}
\setcounter{equation}{0}

 In this paper we study  actions of the two-variable Hecke operators
 on different function spaces 
 and determine their commutation relations relative  to various other operators appearing in parts I-III.
 
%
%
%
\subsection{Two-variable Hecke operators and $\hR$-operator }\label{sec21}

 In Section \ref{sec3} we introduce and study  functional operators and  differential operators  on
 functions on {the open unit square} $\Box^{\circ} = (0,1)^2$.
The first of these operators is  the 
$\hR$-operator defined by 
$$
\hR (F) (a, c) = e^{-2\pi i a c} F(1-c, a). 
$$
{ As noted earlier, this  operator  satisfies $\hR^4 = I$; 
it plays a role analogous to the Fourier transform. 
The functional equations
for the Lerch zeta function given in \eqref{fun1} and \eqref{fun2} 
can be put in an elegant form using the $\hR$-operator, see  Proposition \ref{pr21}.}

The second operator $D_L$
is constructed using the linear partial differential operators 
 $D_L ^{+} = \frac{\partial}{\partial c}$ and $D_L^{-} = \frac{1}{2\pi i} \frac{\partial}{\partial a} + c$
 for which one has  
\begin{eqnarray} \label{105}
(D_L ^{+}\zeta)(s, a, c) &=& -s \zeta(s+1, a,c)\nonumber \\
&~&\\
(D_L ^{-}\zeta)(s, a, c) &=& \zeta(s-1, a,c). \nonumber
\end{eqnarray}
{ These operators define unbounded operators
on $L^2( \Box, da\, dc)$ and on $L^1(\Box, da\, dc)$.
In the remainder of Section \ref{sec3} we determine   commutation relations among these operators} 
on the space $C^{1,1} ( \Box^{\circ})$ { of jointly continuously differentiable functions   
on the open unit square $\Box^{\circ}$ 
whose
mixed second partials  are continuous and on which the two first partials commute.}

In Section \ref{sec4} we construct the ``real-variables'' action 
of the two-variable Hecke operators on an extended function space.
We also consider the Hecke operators conjugated by powers of $\hR$:  
$$
\hS_m  :=  \hR \hT_m \hR^{-1} ~,~{\hT}_m^{\vee}  :=  \hR^2 \hT_m \hR^{-2},~~ 
\mbox{and}~~{\hS}_m^{\vee}  :=  \hR^3 \hT_m \hR^{-3}.
$$
Together with $\hT_m$ we have four families of operators, each of which  formally leaves invariant a line passing
through one side of the square $\Box$.


\begin{theorem}\label{th2001}
{\rm (Commuting Operator Families)} 

(1) The four sets of two variable Hecke operators $\{ \hT_m, \hS_m, \hT_m^{\vee}, \hS_m^{\vee}:\,  m \ge 1\}$
{continuously extend to   bounded operators  on each Banach space $L^{p}(\Box, da \, dc)$} 
for $1 \le p \le \infty$, by viewing these (almost everywhere defined)
functions of $\Box$ as extended to $\RR \times \RR$ via the  twisted-periodicity relations
$ f(a+1, c) = f(a, c)$ and $f(a, c+1) = e^{-2\pi i a} f(a, c)$.
These operators satisfy  $\hT_m = \hT_m^{\vee}$, $\hS_m = \hS_m^{\vee}$ and $\hS_m = \frac{1}{m}(\hT_m)^{-1}$ for all $m \ge 1$.

(2)  The $\mathbb C$-algebra $\sA_{0}^{p}$ of operators on $L^p(\Bx, da \, dc)$ generated by these four sets of operators under addition and operator multiplication is commutative.

(3) On $L^2(\Box, da \, dc)$  the adjoint Hecke operator $(\hT_m)^{\ast} = \hS_m$, and 
 $(\hS_m)^{\ast}= \hT_m$. In particular the $\CC$-algebra $\sA_0^2$
is a $\star$-algebra. { In addition each of $\sqrt{m}\hT_m, \sqrt{m}\hS_m$ is 
a unitary operator on  $L^2(\Box, da \, dc)$.}
\end{theorem}

For all $p \ge 1$ the operator $\hR$ defines  
an isometry  $||\hR (f)||_p = ||f||_p$ on $L^p(\Bx, da \, dc)$.
In consequence we obtain an extended algebra
$\sA^p := \sA_0^p [ \,\hR\,]$ of operators, by adjoining the operator $\hR$ to $\sA_0^p$.
The algebra  $\sA^p$ is a noncommutative algebra. {The algebra  $\sA^2$ is also  a $\star$-algebra, 
which follows using $(\hR)^{\star} = \hR^{-1}= \hR^3.$}

%
%
%
\subsection{Two-dimensional  Lerch eigenspace}\label{sec22}

In Section \ref{sec5} we study the Lerch eigenspace
$\sE_s$ defined as  the vector space
over $\CC$ spanned by the four functions
$$
\sE_s := <L^{+}(s, a, c), L^{-}(s, a, c),  e^{-2 \pi i a c} L^{+}(1-s, 1-c, a), e^{-2 \pi i a c} L^{-}(1-s, 1-c, a)>, 
$$
viewing the $(a, c)$-variables on $\Bxc$.
{  We treat the variable
$s$ as constant, and all $s \in \CC$ are allowed since on $\Bxc $ the functions $L^{\pm}(s, a, c)$
 are entire functions of $s$. Note that the gamma factors are omitted from the functions  in this definition.}
These functions satisfy linear dependencies by virtue of the two functional equations
that $L^{\pm}(s, a, c)$ satisfy { (Theorem \ref{th201}).} 

\begin{theorem}\label{th2002} {\rm (Operators on Lerch eigenspaces)}
For each $s \in \CC$ the space $\sE_s$ is
a two-dimensional vector space.
All functions in $\sE_{s}$ have the following properties.
\begin{itemize}
\item[(i)] 
{\rm (Lerch differential operator eigenfunctions)}
Each $f \in \sE_s$ is an eigenfunction of the Lerch differential operator 
$\hD_L = \frac{1}{2 \pi i} \frac{\pt}{\pt a}
\frac{\pt}{\pt c} + c \frac{\pt}{\pt c}$ with
eigenvalue $-s$, namely 
$$
(\hD_L f )(s, a, c) = -sf (s, a, c) 
$$
holds at all $(a, c) \in \RR \times \RR$, with both $a$ and $c$ non-integers. 
\item[(ii)] 
{\em (Simultaneous Hecke operator eigenfunctions)}
Each $f \in \sE_s$ is a simultaneous eigenfunction { with
eigenvalue $m^{-s}$ of all two-variable Hecke operators
$$
\hT_m (f)(a,c) = \frac{1}{m} \sum_{k=0}^{m-1} f \left( \frac{a+k}{m}, mc \right)
$$
in the sense that, for each $m \ge 1$,}
$$
\hT_m f = m^{-s} f  
$$
holds  on the domain $(\RR \smallsetminus \frac{1}{m}\ZZ)
\times (\RR \smallsetminus \ZZ)$.
\item[(iii)] 
{\rm ($\hJ$-operator eigenfunctions)} The space $\sE_s$ admits the involution
{ $$\hJ f (a,c) : = e^{-2 \pi ia} f(1-a, 1-c),$$} under which it 
 decomposes into one-dimensional eigenspaces 
$\sE_s = \sE_s^+ \oplus \sE_s^-$ with eigenvalues $\pm 1$,
that is, $\sE_s^{\pm} = <F_{s}^{\pm}>$ and  
$$
\hJ (F_s^{\pm}) = \pm F_s^{\pm}.
$$
\item[(iv)]  
{\rm ($\hR$-operator action)} 
The $\hR$-operator $\hR (F) (a, c) = e^{-2\pi i a c} F(1-c, a)$
acts by $\hR(\sE_s)= \sE_{1-s}$ with
$$
\hR(L^{\pm}(s, a, c) ) = w_{\pm} ^{-1}\ggG^{\pm}(1-s) L^{\pm}(1-s, a, c), 
$$
where $w_{+}= 1$, $w_{-}=i$, $\gamma^+(s) = \Gamma_{\RR}(s)/\Gamma_{\RR}(1-s)$, $\gamma^-(s) = \gamma^+(s+1)$, and $\Gamma_{\RR}(s) = \pi^{-s/2}\Gamma(s/2)$. 
\end{itemize}
\end{theorem}

The  members of the Lerch eigenspaces are shown to 
have the following analytic properties { (Theorem \ref{th31}).}


\begin{theorem}\label{th2003}
{\rm (Analytic Properties of Lerch eigenspaces)}
For fixed $s \in \CC$ the functions in the Lerch eigenspace $\sE_s$ are real analytic
functions of $(a, c)$ on  $(\RR \smallsetminus \ZZ) \times (\RR \smallsetminus \ZZ)$,
which may be discontinuous at values $a, c \in \ZZ$. They 
have the following properties.
\begin{itemize}
\item[(i)] {\rm (Twisted-Periodicity Property)} All functions  $F(a,c)$ in $\sE_{s}$  satisfy the 
twisted-periodicity functional equations
\begin{eqnarray*} 
F(a+1,~c~) & =& ~~~~~~~~ F(a,c), \\
F(~a~, c+1)& = & e^{-2\pi i  a} F(a, c). 
\end{eqnarray*}

\item[(ii)] {\rm (Integrability Properties)}

(a) If $\Re(s) >0$,  then 
for each noninteger $c$ all functions  in $\sE_{s}$
have 
$f_c(a) :=F(a, c) \in L^{1}[(0,1), da],$ and  
all  their Fourier coefficients
$$
f_n(c) := \int_{0}^1 F(a,c) e^{-2\pi i n a} da, ~~~~~~~n \in \ZZ,
$$
are  continuous
functions of $c$ on $0<c<1$.

(b) If $\Re(s)<1$,  then for each noninteger $a$ all functions  in $\sE_{{s}}$
have $ g_a(c):=e^{2 \pi i a c}F(a, c) \in L^{1}[(0,1), dc],$
and all Fourier coefficients
$$
g_n(a) := \int_{0}^1 e^{2\pi i a c}F(a,c) e^{-2\pi i n c} dc, ~~~~~~~n \in \ZZ,
$$
 are  continuous
functions of $a$ on $0< a<1$.

(c) If $0 < \Re(s) < 1$ then all functions in $\sE_{s}$ belong to 
$L^{1}[\Box , da dc]$. In this range the vector  space $\sE_s$ is  invariant under the action
of all four sets $ \hT_m, \hS_m, \hT_m^{\vee}, \hS_m^{\vee}  ~( m \ge 1)$ of two-variable Hecke operators. 
\end{itemize}
\end{theorem}

%
%
%
\subsection{Eigenfunction characterization of Lerch eigenspace}\label{sec23}

In Section \ref{sec6} we first review  Milnor's simultaneous eigenfunction characterization of
Kubert functions (Theorem \ref{th41}).  We then prove  the 
 following  characterization for the Lerch eigenspace
$\sE_s$ (Theorem \ref{th62}), { the main result of this paper}, which can be viewed as a generalization to two dimensions
of Milnor's result. { It may also be regarded as a converse theorem to Theorem \ref{th2003}.}


\begin{theorem}~\label{th2004}
{\rm (Lerch Eigenspace Characterization)}
Let $s\in \CC$. Suppose that
$F(a, c): (\RR \smallsetminus \ZZ) \times (\RR \smallsetminus \ZZ) \to \CC$ 
is a continuous function  that satisfies the following conditions.
\begin{itemize}
\item[(1)] {\em (Twisted-Periodicity Condition) }
For $(a,c) \in (\RR \smallsetminus \ZZ) \times (\RR \smallsetminus \ZZ)$, 
\begin{eqnarray*}
F(a+1, ~c~) & =& ~~~~~~~~F(a,c) \\
F(a, ~c+1) &=& e^{-2\pi i a} F(a,c).
\end{eqnarray*}

\item[(2)] {\em (Integrability Condition) } At least one of the following two
conditions (2-a) or (2-c) holds.

{\rm (2-a)} The $s$-variable has
$\Re(s) > 0$. 
For   $0<c<1$ each function $f_c(a):= F(a, c) \in L^1[(0,1), da]$, and    
  all  the Fourier coefficients 
 \[
   f_n(c) := \int_{0}^{1} f_c(a)e^{-2\pi i n a} da= \int_{0}^1 F(a,c) e^{-2\pi i n a}da,~~~n \in \ZZ,
   \]
  are continuous functions of $c$.

 {\rm (2-c)}  The $s$-variable has
 $\Re(s) < 1$. For  $0<a<1$ each function $g_a(c):= e^{2 \pi i a c}F(a,c) \in L^1[(0,1), dc]$,  and
 all  the Fourier coefficients 
 \[
   g_n(a) := \int_{0}^{1} g_a(c)e^{-2\pi i n c} dc= \int_{0}^1 e^{2\pi i a c}F(a,c) e^{-2\pi i n c}dc,~~~n \in \ZZ,
   \]
 are continuous functions of $a$.

  \item[(3)] {\em (Hecke Eigenfunction Condition)} For all $m \ge 1$, 
$$
\hT_{m}(F)(a, c) = m^{{-s}}F(a, c)
$$
holds for  
$(a, c) \in (\RR \smallsetminus \ZZ)\times (\RR \smallsetminus \frac{1}{m}\ZZ)$.
     \end{itemize} 
    
\noindent Then  
$F(a, c)$ is the restriction to noninteger $(a, c)$-values
of a function in the Lerch eigenspace $\sE_{s}$.
\end{theorem}

This generalization  imposes extra integrability conditions in order to control
the dilation property of the two-variable Hecke operators in the $c$-coordinate.
{  The twisted-periodicity hypothesis plays an essential role in
the proof of Theorem \ref{th2004},  in giving identities  that the Fourier series coefficients of
such functions must satisfy, see \eqref{440}.}

%
%
%
\subsection{{Further extensions}}\label{sec24}

One may further { define generalizations} of the Lerch zeta function that incorporate
 Dirichlet characters $\chi(\bmod~N)$. In the real-analytic case, they  take the form
 $$
 L^{\pm}(\chi, s ,a ,c ) = \sum_{n \in \ZZ} \chi(n) (\sgn (n+c))^k e^{2\pi i na} |n+c|^{-s},
 $$
 with $k=0,1$. These functions  can be expressed as linear combinations of scaled
 versions of $L^{\pm}(s, a,c).$
 The Hecke operator framework  extends  to apply to such
functions, with some extra complications, see \cite{Lag15}.

  Subsequent work of  the first author \cite{Lag15} 
gives a representation-theoretic interpretation of the
``real-variables'' version of the Lerch zeta function treated in this paper, 
related to function spaces
associated to the Heisenberg group.
Under it, various combinations of Lerch zeta function and variants  twisted
by Dirichlet characters are found to play the role of Eisenstein series
with respect to the operator $\Delta_L=D_L + \frac{1}{2} {\bf I}$, 
where ${\bf I}$ is the identity operator, 
which plays the role of a Laplacian.
The Lerch zeta functions  treated in this paper correspond to the 
trivial Dirichlet character. In this Eisenstein series interpretation the Lerch
functions $L^{\pm}(s, a, c)$ on the line $\Re(s)= \frac{1}{2}$
parametrize a pure continuous spectrum of the operator $\Delta_L$
on a suitable Hilbert space $\sH$. The  real-analytic  Hecke operators $\hT_m$ described in this paper 
then correspond to dilation operators similar to $\tilde{\hT}_m$ in Section \ref{sec13}.

In another sequel paper \cite{La5} the first author  studies
a  ``complex-variables'' two-variable 
Hecke operator action associated to  the Lerch zeta functions.
  This framework again includes the Lerch zeta function for $\Re(s)>1$,
with two-variable Hecke operators initially viewed  as acting on a suitable domain of holomorphic functions.   
It then  studies these Hecke operators acting on spaces of (multivalued)
holomorphic functions on various domains. 
These are quite  different function spaces than the ones treated here
and the resulting Hecke action has new features. In particular members of different  families of Hecke operators 
 do not commute on these function spaces.   \medskip

\paragraph{\bf Notation.} Much of the analysis of this paper concerns functions
with $s \in \CC$ regarded as a parameter. We let $\Bx= \{ (a, c) \in [0, 1]\times [0, 1]\}$,
and its interior
$ \Bxc = \{ (a,c) \in (0,1) \times (0,1)\}$.

%
%
%
\section{$R$-Operator and Differential Operators}\label{sec3}
\setcounter{equation}{0}

We introduce various operators with respect to which  
the Lerch zeta function has  invariance properties.
We first restate  operators acting on functions 
of two real variables $(a, c)$ defined almost everywhere
on the unit square $\Box^{\circ}$. 
Then we consider operators defined on larger domains 
 of $\RR \times \RR$. 
We start with  a non-local  operator $\hR$ mentioned before, associated
to the Fourier transform, and certain
linear partial differential operators, which we term real-analytic Lerch differential operators.
 In \S3.1 we define the $\hR$-operator and
reformulate  the functional equations for the Lerch zeta function  in part I,
 using this operator. In \S3.2  
we define Lerch differential  operators and  in \S3.3 we determine  their  commutation relations
 on suitable function spaces, together with the $\hR$-operator.

%
%
%

\subsection{Local operators: Real-variable  differential  operators} \label{sec31}

We first study transformation of the Lerch zeta function under certain 
(real variable) linear
partial differential operators. {Recall from \S2.1 the operators}  
\beql{N201}
\hD_L^{+} = \frac{\partial}{\partial c} , 
\eeq
called the  ``raising operator'', and 
\beql{N202}
\hD_L^{-}= \frac{1}{2 \pi i}  \frac{\partial}{\partial a} + c
\eeq
{ called} the  ``lowering operator''. 
The names are suggested by the property that
for $\Re(s)>1$ the Lerch zeta function satisfies  the recurrence equations
\beql{N108}
(\hD_L^{+}\zeta)(s, a, c) = -s \zeta(s+1, a, c)
\eeq
and
\beql{N109}
(\hD_L^{-} \zeta)(s, a, c) = \zeta(s-1, a, c),
\eeq
which raise and lower the value of the $s$ parameter. 
These identities follow 
using its series definition \eqn{101}.

We combine the two operators to obtain  the 
{\em Lerch differential operator} $\hD_L$ { as in (\ref{DL}),} given by
\beql{N110}
\hD_L = \DLm \DLp =
 \left( \frac{1}{2 \pi i}  \frac{\partial}{\partial a} + c\right) \frac{\partial}{\partial c}.
\eeq
{It follows from (\ref{N108}) and (\ref{N109}) that} the Lerch zeta function is an  eigenfunction of the Lerch differential operator
for $\Re(s) >1$, satisfying 
\beql{N111}
(\hD_L \zeta)(s, a, c) = -s\zeta(s, a, c),
\eeq
as noted in \eqref{Eigenvalue}.


\begin{lemma}\label{lem31}
The operator $\hD_L$ 
takes a 
piecewise $C^{1,1}$-function
that is twisted-periodic on $\RR \times \RR$ to 
a twisted-periodic function  on $\RR \times \RR$.
\end{lemma}

\begin{proof}
For piecewise  $C^{1,1}$-functions, the derivatives $\frac{\partial}{\partial a}$
and $\frac{\partial}{\partial c}$ commute. {(We allow discontinuities at the
edges of the pieces, with the boundaries being smooth curves; straight line boundaries
are used in the sequel.)}
By definition, a twisted-periodic function $F$ on $\RR \times \RR$ satisfies 
$F(a+1, c) = F(a, c)$ and  $F(a, c+1)= e^{-2\pi i a} F(a, c)$. We have to show these two relations with $F$ replaced by $D_LF$,
whenever partial derivatives are allowed.  Here $(D_L F)(a, c)$ means
$$ (D_L F)(a, c) = \big( \frac{1}{2\pi i} \frac{\partial}{\partial a} \frac{\partial}{\partial c} + c \ \frac{\partial}{\partial c} \big)(F(a, c)).$$
Note that $\frac{\partial}{\partial a} = \frac{\partial}{\partial (a+1)}$ and $\frac{\partial}{\partial c} = \frac{\partial}{\partial (c + 1)}$ by {the} chain rule.
It is easy to check that
$(\hD_L  F)(a+1, c) = (\hD_L F)(a, c)$. For the remaining relation,
using twisted-periodicity gives
\begin{eqnarray*}
(\hD_L F)(a, c+1) &=& \big( \frac{1}{2\pi i} \frac{\partial}{\partial a} \frac{\partial}{\partial c} + (c+1) \ \frac{\partial}{\partial c} \big)(F(a, c+1))\\
&=& \big( \frac{1}{2\pi i} \frac{\partial}{\partial a} \frac{\partial}{\partial c} + (c+1) \ \frac{\partial}{\partial c} \big)(e^{-2 \pi i a} F(a, c))\\
&=& e^{-2 \pi i a} \frac{1}{2\pi i} \frac{\partial}{\partial a} \frac{\partial}{\partial c} (F(a, c)) + e^{-2 \pi i a}(-1+ c+1) \ \frac {\partial}{\partial c}( F(a, c))\\
&=& e^{-2 \pi i a} (\hD_LF)(a, c), 
\end{eqnarray*}
as required.
\end{proof}

%
%
%

\subsection{Non-local operator: $\hR$-operator }\label{sec32}

{Recall from (\ref{R}) the  {\em $\hR$-operator} which} 
acts on suitable functions  on the  open square 
$\Box^{\circ}$ 
given by
\beql{N105}
\hR( f)(a, c)\,\,\,\,  = \quad e^{-2 \pi i  a c} f(1-c , a).\quad\quad
\eeq
This operator  satisfies $\hR^4=\hI$,
with
\begin{eqnarray}\label{N105b}
\hR^2 (f)(a,c) & = & e^{-2 \pi ia} f(1-a , 1-c), \\
\label{106}
\hR^3( f)(a,c) & = & e^{-2 \pi iac + 2 \pi ic} f (c, 1-a),\\
\hR^4( f)(a,c) & = & f (a, c).
\end{eqnarray}
{Note that the operator $J$ in Theorem \ref{th2002} is nothing but $\hR^2$, which we sometimes term}  the  {\em  reflection operator} since it 
relates the function values in each coordinate  about the point $a= \frac{1}{2}$, resp.  $c= \frac{1}{2}$.

These operators have   well-defined actions on continuous functions $C^{0}( (0,1)^2)$  defined in
the open unit square. They extend by closure to an operator action on 
(almost everywhere defined) $L^p$-functions  on the unit square,
for any $p \ge 1$.  In particular, the extended action on $L^2 ( \Box, da dc)$ is unitary.


\begin{lemma}\label{lem32}
The operator $\hR$ preserves the property of being (almost everywhere) twisted-periodic
on $\RR \times \RR$.
\end{lemma}

\begin{proof}
Given a twisted-periodic $F(a, c)$ on $\RR \times \RR$, set $g(a, c) = \hR(F)(a, c)= e^{-2 \pi i ac}F(1-c, a)$.
Then
\begin{eqnarray*}
g(a, c+1) &=& e^{- 2 \pi i a (c+1)} F( 1- (c+1), a)\\
&=& e^{- 2\pi i ac} e^{-2 \pi i a} F(-c, a)\\
&=& e^{- 2 \pi i a} (e^{-2 \pi i ac} F(1-c, a)) = e^{-2 \pi i a} g(a, c)
\end{eqnarray*}
where twisted-periodicity of $F$ was used on the third line. Then
\begin{eqnarray*}
g(a+1, c) &=& e^{- 2 \pi i (a+1)c} F( 1- c, a+1)\\
&=& e^{- 2\pi i ac} e^{-2 \pi i c} F(1-c, a+1)\\
&=&  e^{-2 \pi i ac} e^{- 2\pi i c}(e^{2 \pi i c} F(1-c, a)) =  g(a, c),
\end{eqnarray*}
where twisted-periodicity of $F$ was used on the third line.
\end{proof}
\begin{rem}\label{remark31}
{\em 
The $\hR$-operator can be extended to act  pointwise on  functions defined almost everywhere
on $\Box^{\circ}$, with the same rule \eqref{N105}. That is, we can allow singularities on a
finite set of horizontal and vertical lines. 
It can also be extended to the domain $\RR \times \RR$, with the same rule \eqref{N105}, acting on functions
satisfying a twisted-periodic relation, see Lemma \ref{le45}.
}
\end{rem}

%
%
%

\subsection{Commutation relations} \label{sec33}

Let $C^{1,1}(\Box^{\circ})$
denote the complex vector space of jointly continuously differentiable functions $f(a,c)$  
on the open unit square $\Box^{\circ}$ 
whose
mixed second partials  are continuous functions and satisfy
$$\frac{\partial^2 f}{\partial a\partial c}(a, c) =
\frac{\partial^2 f}{\partial c\partial a}(a, c).
$$

The following result describes commutation relations
of these operators among themselves and with the $\hR$-operator.


\begin{lemma}\label{le301}
On the space $\rC^{1,1} (\Box^{\circ})$ the 
following properties hold.
\begin{itemize}
\item[(i)]
The operators $\DLp = \frac{\pt}{\pt c}$ and $\DLm= 
\frac{1}{2 \pi i} \frac{\pt}{\pt a} + c$ satisfy the commutation relation 
\beql{303}
\DLp \DLm - \DLm \DLp= I ~,
\eeq
{ where $I$ denotes the identity operator.}
\item[(ii)]
The operator $\hR f (a,c) := e^{- 2 \pi i ac} f(1-c,a)$ leaves $\rC^{1,1} (\Box^{\circ})$ 
invariant and satisfies
\begin{eqnarray}\label{304}
\DLp \hR  &= &- 2 \pi i ~\hR \DLm ~, \\
\label{305}
\DLm \hR &= &\frac{1}{2 \pi i} \hR \DLp ~.
\end{eqnarray}
\item[(iii)]
The operator $\hD_L = \DLm \DLp$ satisfies
\beql{306}
\hD_L \hR + \hR \hD_L = - \hR~ .
\eeq
In particular, $\hD_L$ and $\hR^2$ commute: 
\beql{307}
\hD_L \hR^2 = \hR^2 \hD_L ~.
\eeq
\end{itemize}
\end{lemma}

\begin{proof}
These results follow by direct  calculation. 
For (i), we have 
$$
\DLp \DLm  = \frac{\pt}{\pt c} ( \frac{1}{2 \pi i} \frac{\pt}{\pt a} + c ) 
= \frac{1}{2 \pi i} \frac{\pt}{\pt a} \frac{\pt}{\pt c} + 
c \frac{\pt}{\pt c} + I= \DLm \DLp +I. 
$$

For (ii), we use $\hD_1$, $\hD_2$ to denote partial derivatives with respect to 
the first and second variable, respectively, in order to avoid
ambiguity under the action of $R$, which interchanges the variables.
Let
$$
h(a,c) := ( \DLm f ) (a,c) = \left( \frac{1}{2 \pi i} 
\hD_1f  + cf \right)(a,c) ~.
$$
Then
\begin{eqnarray*}
(\hR \DLm f ) (a,c) & = & (\hR h) (a,c) = e^{-2 \pi iac} h(1-c,a) \\
& = & e^{-2 \pi iac} 
\left( \frac{1}{2 \pi i} (\hD_1 f) (1-c,a) + a f (1-c, a) \right) ~.
\end{eqnarray*}
On the other hand, let
$$g(a,c) = f(1-c,a)$$
so that
$(\hR f) (a,c) = e^{-2 \pi iac} g(a,c) ~.$
We have
\begin{eqnarray*}
(\DLp \hR f) (a,c) & = & - 2 \pi ia e^{-2 \pi iac} g(a,c) + 
e^{-2 \pi iac} (\hD_2 g) (a,c) \\
& = & -2 \pi ia 
e^{-2 \pi iac} f(1-c,a) - e^{-2 \pi iac} (\hD_1 f) (1-c,a) \\
& = & -2 \pi i (\hR \DLm f ) (a,c) ~.
\end{eqnarray*}
In other words,
$\DLp \hR = - 2 \pi i \hR \DLm~, $ which is \eqref{304}.

Next, we have
\begin{eqnarray*}
(\DLm \hR f) (a,c) & = & \frac{1}{2 \pi i}
\frac{\pt}{\pt a} (\hR f) (a,c) + c(\hR f) (a,c) \\
& = & \frac{1}{2 \pi i} (-2 \pi ic e^{-2 \pi iac} f(1-c,a) 
+ e^{- 2 \pi iac} \hD_1 g(a,c)) + ce^{-2 \pi iac} f(1-c,a) \\
& = & \frac{1}{2 \pi i} e^{-2 \pi iac} \hD_2 f (1-c,a) =
\frac{1}{2 \pi i} (\hR \DLp f) (a,c) ~.
\end{eqnarray*}
In other words,
$\DLm \hR = \frac{1}{2 \pi i} \hR \DLp ~,$ which is \eqref{305}.

For (iii) we use  (i) and (ii) to obtain
\begin{eqnarray*}
\hD_L \hR & = & \DLm \DLp \hR = - 2 \pi i \DLm \hR \DLm
= - \hR \DLp \DLm \\
& = & - \hR (\DLm \DLp + I ) =  - \hR \hD_L - \hR ~,
\end{eqnarray*}
which is \eqref{306}. This identity yields
$$
\hD_L \hR^2 =  (-\hR \hD_L -\hR) \hR = - \hR(\hD_L \hR + \hR) =
-\hR (-\hR \hD_L) = \hR^2 \hD_L~,
$$
which is \eqref{307}.
\end{proof}

%
%

\section{Two-variable Hecke Operators}\label{sec4}
\setcounter{equation}{0}

The individual terms in the sums defining two-variable Hecke operators \eqref{102} 
are  operators that act on the domain by affine changes of variable.
They are compositions of $a$-variable operators of form
$$
\hT_{k, m}(f) (a, c) :=  f( \frac{a+k}{m}, c),
$$
which map the unit square into itself, together with the 
 the $c$-variable dilation operator
$$
D_m^c(f)(a, c) := f(a,mc),
$$
which maps the unit square outside itself. 
A main problem
in obtaining a well-defined action of the two-variable Hecke operators \eqref{102} 
is to accommodate the ``dilation'' action in the $c$-variable,  which
expands the domain. In the $c$ variable the domain must therefore include at least
$\RR_{>0}$ or $\RR_{<0}$. 
The additive action in the $a$-variable
requires to handle all $\hT_{k,m}$ for $0 \le k < m$  a ``non-local'' definition valid over an interval of width at least one.

The viewpoint of this paper is start with  functions defined (almost everywhere) on the  unit square $\Box= [0,1] \times [0,1]$,
extend them to functions defined on the plane $\RR \times \RR$ by imposing  twisted-periodicity conditions:
$f(a+1, c) = f(a, c)$ and $f(a, c+1) = e^{-2 \pi i a} f(a, c)$, and use the extended functions to define an
action of the two-variable Hecke operators on functions on the unit square; the resulting action is linear and preserves
twisted-periodicity conditions on the functions. 

We would like to accommodate  functions with  discontinuities 
of the type that naturally appear in connection with the Lerch zeta function
in the  real-analytic framework, which occur for integer values of $a$ and/or $c$. 
We must also deal with  the problem that the  action of $\hT_m$ changes the
location of the discontinuities.
We will introduce  appropriate  function spaces that are  closed under the Hecke operator action
allowing a finite (but variable) number of discontinuities in horizontal and vertical
directions. These operators also make sense on the Banach space 
 $L^1( \Box, da\, dc)$ and on the Hilbert space $L^2(\Box, da\,dc)$.

%
%

\subsection{Twisted-periodic function spaces}\label{sec41}

In the  real-variables framework treated in this paper we take as a function
space the set of   piecewise-continuous
functions with a finite number of pieces defined in the open unit square   $\Box^{\circ}$,
allowing discontinuities on horizontal and vertical lines, and permitting the number
of discontinuities to depend on the function.

\begin{defi}~\label{Nde31}
{\em
 (i) The {\em  twisted-periodic function space} $\sP(\Box^{\circ})$ 
is the  complex vector space that consists   of equivalence classes of functions
$f: \Box^{\circ} \to \CC$
which are piecewise continuous, with discontinuities, if any,  located
along a finite number of horizontal lines at rational coordinates $c= \frac{j}{d}, ~0 < j <d$,
for some integer $d \ge 1$, called an {\em admissible denominator} for the function. 
We then extend these functions  to  functions
 $f: (\RR \smallsetminus \ZZ) \times (\RR \smallsetminus \ZZ)$,
by imposing  the twisted-periodicity conditions
\begin{eqnarray} 
f(a+1, c) & = & f(a, c) \label{N311a} \\
f(a~~, c+1) & =& e^{-2 \pi i a} f(a, c) \label{N311b}
\end{eqnarray}
and  denote the set of such equivalence classes of functions $\sP(\RR \times \RR)$.
Here two functions with admissible denominators $d$, $d'$ are considered {\em equivalent}, if their
values coincide on  the set 
$$
\sS_{dd'} := \{ (a, c) \in (\RR \smallsetminus \ZZ) \times (\RR \smallsetminus \frac{1}{d d'}\ZZ)\}.
$$

(ii) The smallest lattice $\frac{1}{d}\ZZ$  on which discontinuities occur (for $d \ge 1$) 
will be called the 
{\em support lattice} of the function, and the minimal value $d$ can be called the
{\em conductor} of the function.}
\end{defi}

We allow  ``functions'' in $\sP(\RR \times \RR)$ to  be undefined on the set of horizontal lines with
coordinates $c = \frac{j}{d}$ for all $j \in \ZZ,$ and don't compare their values at such points.
Note that a function having {admissible} denominator $d$ can be regarded also a function with { admissible} denominator
$kd$ for any integer $k \ge 1$. Aside from this ambiguity, any function in $\sP(\RR \times \RR)$
is completely determined by its values in $\sP(\Box^{\circ})$, { using the twisted- periodicity conditions. 
More precisely, the equivalence relation on  ``functions" in $\sP(\RR \times \RR)$
calls functions equivalent  whose values agree outside of a set of measure zero;
the space   $\sP(\RR \times \RR)$ has a well-defined vector space structure on equivalence classes. }

%
%

\begin{prop}~\label{Nle31} {\rm (Real Variable Hecke Operators)}
The two-variable Hecke operators
$$
\hT_m(f)(a, c) = \frac{1}{m} {\sum_{k=0}^{m-1} f \big(\frac{a+k}{m}, mc \big) }
$$
act on twisted-periodic functions in  $ \sP(\RR \times \RR)$, 
i.e. they preserve twisted-periodicity. On this function space
they satisfy
\beql{N315}
\hT_m  \hT_n = \hT_n   \hT_m = \hT_{mn}.
\eeq
\end{prop}

\begin{proof} 
If $f(a, c) \in  { \sP(\RR \times \RR)}$, with admissible denominator $d$, then
each $f( \frac{a+k}{m}, mc) \in { \sP(\RR \times \RR)}$ with admissible denominator $md$, hence
$\hT_m(f) \in { \sP(\RR \times \RR)}$ with admissible denominator $md$. 

It remains to verify \eqn{N315}. Assuming that $f (a, c)$ is twisted-periodic 
with discontinuities at level  denominator $d$, the following
calculation is well-defined with denominator $mnd$, 
\begin{eqnarray}
\hT_m (\hT_n (f) )(a,c) & = & \frac{1}{m} \sum_{k=0}^{m-1} (\hT_n( f))
\left( \frac{a+k}{m} , mc \right)  \nonumber \\
& = & \frac{1}{m} \sum_{k=0}^{m-1} \left(  \frac{1}{n}
\sum_{l=0}^{n-1} f \big( \frac{\frac{a + k}{m} + l}{n} , n(mc) \big) \right)   \nonumber \\
& = & \frac{1}{mn} \sum_{k=0}^{m-1} 
\left( \sum_{l=0}^{n-1} f \big(\frac{a+k+lm}{mn} , mnc \big) \right)  \nonumber \\
& = & \hT_{mn} (f ) (a,c)   \label{N316} \,.
\end{eqnarray}
The equality $\hT_n(\hT_m)(f)= \hT_{mn}(f)$ follows similarly.
\end{proof}

%
%

\begin{rem}\label{rem43}
{\em
One may also study 
two-variable Hecke operators $\hT_m$ and $\hR$ acting as bounded operators on the 
 function spaces $L^1(\Box, da\, dc)$ or
$L^2(\Box, da \,dc)$. Here we view the functions as extended to (most of) $\RR \times \RR$
by twisted-periodicity to define the Hecke operator action. 
We note that  the twisted-periodicity conditions in Definition \ref{Nde31} preserve both the
$L^1$-norm and the $L^2$-norm of functions on unit squares having  integer lattice points as corners.
}
\end{rem}

%
%

\subsection{$\hR$-operator  conjugates of Hecke operators}\label{sec42}

Conjugation by powers of $\hR$ of the
Hecke operators gives rise to four related 
families of two-variable Hecke operators. The first family consists of  
the operators $\{\hT_m: m \ge 1\}$ in \eqn{102}
and the remaining three families are
\beql{N331}
\hS_m  :=  \hR \hT_m \hR^{-1} ~,~{\hT}_m^{\vee}  :=  \hR^2 \hT_m \hR^{-2},~~ 
\mbox{and}~~{\hS}_m^{\vee}  :=  \hR^3 \hT_m \hR^{-3}.
\eeq
Their actions are given by 
\begin{eqnarray}
{\hS}_m f (a,c) & := & \frac{1}{m} \sum_{k=0}^{m-1} 
e^{2 \pi i ka} f \left( ma, \frac{c+k}{m} \right); \label{hS}\\
{\hT}_m^{\vee} f (a,c) & := &
\frac{1}{m} \sum_{k=0}^{m-1} 
e^{2\pi i (\frac{(1-m)a+k}{m})}f \left( \frac{a+k}{m} , 1 + m(c - 1) \right); \label{hT} \\
{\hS}_m^{\vee} f (a,c) & := & \frac{1}{m} \sum_{k=0}^{m-1} 
e^{2 \pi i (m-(k+1) )a}f \left( 1 + m(a - 1), \frac{c+m-(k+1)}{m} \right). \label{hS-vee}
\end{eqnarray}
Each of these families of operators leaves invariant a line passing
through one side of the square $\Box$, 
 in the sense that the above definitions   (formally) make
 sense as  one-variable operators,  when restricted to
 the invariant line.  For  
$\hT_m$,  ${\hT}_m^{\vee}$, ${\hS}_m$  and  ${\hS}_m^{\vee}$  the
invariant lines are  $c=0$,  $c=1$, $a=0$ and $a=1$, respectively.

In order to have a suitable domain inside $\RR \times \RR$ on which all four
families of operators are simultaneously defined we must use essentially
all of $\RR \times \RR$, because between them the operators are
expansive in the positive and negative $a$-directions  and $c$-directions.
We make the following definition, allowing functions with
discontinuities, for the reasons given above.

\begin{defi}~\label{Nde43}
{\em
  The {\em  extended twisted-periodic function space} $\sP^{\ast}(\RR \times \RR)$ 
is a  complex vector space of  equivalence classes of functions
$f:\Box^{\circ} \to \CC$
which are piecewise continuous, having the property that:
there is a positive
integer $d$ (which may depend  on the function) such that all discontinuities of $f(a, c)$
occur only in the following regions:
\begin{itemize}
\item
 along  horizontal lines at rational coordinates $c= \frac{j}{d}$, $0 \le  j < d$,
 \item
along vertical lines at rational coordinates $a= \frac{k}{d}$, $0 \le k < d$. 
\end{itemize}
We call any such value $d$ an {\em admissible denominator}  for the function. (The functions need not be well-defined along the lines of discontinuity.)

(i) These functions are  extended to  functions
 $f: (\RR \smallsetminus \ZZ) \times (\RR \smallsetminus \ZZ) \to \CC$,
by imposing  the twisted-periodicity conditions
\begin{eqnarray*} 
f(a+1, c) & = & f(a, c)\\
f(a~~, c+1) & =& e^{-2 \pi i a} f(a, c).
\end{eqnarray*}

(ii) Two functions with admissible denominators $d$, $d'$ are considered {\em equivalent}, if their
values coincide on  the set 
$$
\sS_{dd'} := \{ (a, c) \in (\RR \smallsetminus \frac{1}{dd'} \ZZ) \times (\RR \smallsetminus \frac{1}{d d'}\ZZ)\}.
$$
We denote the set of such equivalence classes of functions by $\sP^{\ast}(\RR \times \RR)$.

(iii) Given a function in an equivalence class in $\sP^{\ast}(\RR \times \RR)$,
 any integer 
 $d \ge 1$ for which all discontinuities of some function equivalent to $f(a, c)$ 
are on the lattice $\frac{1}{d}\ZZ^2$  will be called a
{\em support lattice} of the function. The minimal value $d$ will be called the
{\em conductor}  of the function\footnote{One can define a refined notion
which allows discontinuities with denominator $d_1$ in the $c$-direction and
$d_2$ in the $a$-direction. Here  $d$ is the least common multiple $d= [ d_1, d_2]$.
One can also define a refined notation $(d_1, d_2)$ of {\em conductor} in the variables separately.}
}.
\end{defi}

The individual operators $f(a, c) \mapsto \frac{1}{m} f(\frac{a+k}{m}, mc)$ making up the two-variable Hecke operator $\hT_m$
leave the space $\sP^{\ast}(\RR \times \RR)$ invariant, and map 
a function in $\sP^{\ast}(\RR \times \RR)$   with conductor $d$ to one with conductor dividing $md$.

One can extend the  ``real-variables'' framework 
 to all four of these families of Hecke operators, acting on the full twisted-periodic function space, using
conjugation by $\hR^j$.  This assertion is justified by the following result,
in which we regard the operator $\hR$ as acting on
functions with domain values $(a, c) \in \RR \times \RR$ by \eqref{N105}.\smallskip

%
%

\begin{lemma}~\label{le45} 
The extended twisted-periodic function space $\sP^{\ast}(\RR \times \RR)$ 
is invariant under the action of the 
operator $\hR$.
\end{lemma}

\begin{proof}
This follows from Lemma \ref{lem32}. 
\end{proof}

%
%

\subsection{Commutation relations of two-variable Hecke operators}\label{sec43}

We now show some surprising commutation relations; on
suitable function spaces all four sets of
Hecke operators mutually commute, despite their non-commutativity with the $\hR$-operator.
%
%

\begin{lemma}\label{le46}
Let $f$ be a function in the {extended} twisted-periodic function vector space $\sP^{\ast}(\RR \times \RR)$,
and {let} $d \ge 1$ be an integer.
 
(1) For $m \ge 1$, 
$$
\hS_m \circ \hT_{dm} (f)(a, c) = \frac{1}{m} \hT_d( f)(a, c).
$$

(2) For  $d=1$ and all $m \ge 1$, 
$$
\hS_m \circ \hT_m(f) (a, c)  { ~= \hT_m \circ \hS_m(f) (a, c)
=\frac{1}{m} f(a,c) }.
$$
Here $\hT_m$ is invertible and $\hS_m = \frac{1}{m}(\hT_m)^{-1}$, and $S_m$ and $T_m$ commute. 
\end{lemma}

\begin{proof}
 The { linear  operators $\hT_m$ and $\hS_m$} take $\sP^{\ast}(\RR \times \RR)$ into itself. We first let $\ell, m \ge 1$ be arbitrary. 
We compute
\begin{eqnarray*}
S_m( T_{\ell} (f))(a, c) & =& \frac{1}{m} \sum_{k=0}^{m-1} e^{2 \pi i ka} \hT_{\ell}(f)( ma, \frac{c+k}{m}) \\
&=& \frac{1}{m} \sum_{k=0}^{m-1} e^{2 \pi i ka}~ \frac{1}{\ell}  \sum_{j=0}^{\ell-1} f \Big(\frac{ma+j}{\ell}, \ell(\frac{c+k}{m})\Big). \\
\end{eqnarray*}

(1) Now suppose $\ell= dm$. Then 
$$
\hS_m( \hT_{dm} (f))(a, c)  = \frac{1}{dm^2}  \sum_{k=0}^{m-1} e^{2 \pi i k a} 
\Big( \sum_{j=0}^{dm -1} f \big(\frac{1}{d} a + \frac{j}{md}, d(c+k)\big) \Big).
$$
Now we apply twisted-periodicity to obtain
\begin{eqnarray*}
\hS_m( \hT_{dm}(f))(a, c) 
&=& \frac{1}{dm^2} \sum_{j=0}^{dm-1} \Big( \sum_{k=0}^{m-1} e^{2 \pi i ka } e^{-2 \pi i kd(\frac{1}{d} a + \frac{j}{md})} f(\frac{a}{d} + \frac{j}{md}, dc) \Big)\\
&=& \frac{1}{dm^2} \sum_{j=0}^{dm-1} f(\frac{a}{d} + \frac{j}{md}, dc)
\Big( \sum_{k=0}^{m-1} e^{-2\pi i \frac{jk}{m}} \Big)\\
&=& \frac{1}{dm} \sum_{n=0}^{d-1} f(\frac{a}{d} + \frac{n}{d}, dc) \\
&= &\frac{1}{m} \hT_d(f)(a, c).
\end{eqnarray*}
Here we used on the second line the orthogonality relation that $\sum_{k=0}^{m-1} e^{-2 \pi i \frac{jk}{m}}=0$ unless $m$ divides $j$,
and equals $m$ if $m$ divides $j$, and in the latter case we write $m=jn$ in the next sum.

(2) When $d=1$ we have 
$\hS_m(\hT_m(f))(a, c) = \frac{1}{m}\hT_1(f) (a, c).$ { On the other hand, 
\begin{eqnarray*}
\hT_m( \hS_{m}(f))(a, c) 
&=& \frac{1}{m} \sum_{k=0}^{m-1} \hS_m(f)(\frac{a+k}{m}, mc) \\
&=& \frac{1}{m} \sum_{k=0}^{m-1} \frac{1}{m} \sum_{j=0}^{m-1} e^{2\pi i j (\frac{a+k}{m})} f(a+k, \frac{mc+j}{m})\\
&=& \frac{1}{m^2} \sum_{j=0}^{m-1} e^{2 \pi ij(\frac{a}{m})}  f(a, c + \frac{j}{m}) \Big( \sum_{k=0}^{m-1} e^{2\pi i  (\frac{jk}{m})}\Big) \\
&=& \frac{1}{m} f(a, c) = \frac{1}{m} T_1(f)(a,c).
\end{eqnarray*}
Here we used the fact that $f$ is invariant under translation by integers in $a$-variable on the third line, and the same orthogonality relation as in (1) on the fourth line.} 
{ This establishes commutativity acting on equivalence classes of functions.}
Since $\hT_1(f)=f$ is the identity map, we 
have shown that $\hT_m$ is both a left and right inverse to $m \hS_m$, whence
  $\hT_m$ is invertible on $\sP^{\ast}(\RR \times \RR)$. 
We then have $\hS_m= \frac{1}{m}(\hT_m)^{-1}$
on this domain.
\end{proof}

%
%

\begin{lemma}\label{le47}
On the  extended twisted-periodic function vector space $\sP^{\ast}(\RR \times \RR)$ 
the operators
$\hS_m$ and $\hT_{\ell}$ commute for all $\ell \ge 1, m \ge 1.$ 
\end{lemma}

\begin{proof}
We know that the operators  $\hT_{\ell}$ mutually commute and that all the $\hS_{m}$ mutually commute.
By Lemma \ref{le46} we know that all $\hT_{\ell}$ are invertible operators on $\sP(\RR \times \RR)$,
and that $\hS_{m}= \frac{1}{m}(\hT_{m})^{-1}$  are also invertible operators.
Thus
$$
\hS_{m} \circ \hT_{\ell} = \frac{1}{m} (\hT_{m})^{-1}\circ \hT_{\ell} = \frac{1}{m} \hT_{\ell} \circ  (\hT_m)^{-1}=\hT_{\ell} \circ \hS_{m},
$$
as required. 
\end{proof}

Next recall that
\begin{equation}\label{T-check}
\hT_m^{\vee}(f)(a, c) =  \frac{1}{m}\Big( \sum_{k=0}^{m-1} e^{2 \pi i ( \frac{(1-m)a+k}{m})} f(\frac{a+k}{m}, 1+ m(c-1)) \Big)
\end{equation}
and note that \eqref{hS-vee} can be rewritten
$$
{\hS}_m^{\vee} f (a,c) = \frac{1}{m} \sum_{k=0}^{m-1} 
e^{-2 \pi i (k+1) a}f \left( 1 + m(a - 1), \frac{c+k}{m}  \right) ~.
$$
\begin{lemma}\label{le48}
On the extended twisted-periodic function space $\sP^{\ast}( \RR\times \RR)$,
we have, for each $m \ge 1$, that 
$$
\hT_m^{\vee} = \hT_m.
$$
In addition, 
$$
\hS_m^{\vee} = \hS_m.
$$
\end{lemma}

\begin{proof}
(1) Under the assumption that $f$ is a twisted-periodic function, we may rewrite \eqref{T-check} as the identity
\begin{equation} \label{T-check-2}
\hT_m^{\vee}(f)(a, c) = e^{- 2\pi i  a} \frac{1}{m} \sum_{k=0}^{m-1} f(\frac{a+k}{m}, m(c-1)).
\end{equation}
We now  show that
$$
\hT_m^{\vee} \circ \hT_m(f) (a, c) = \hT_{m^2}(f)(a, c)
$$
by calculating
\begin{eqnarray*}
\hT_m^{\vee} ( T_m(f))(a, c) &=& e^{-2 \pi i a} \frac{1}{m} \sum_{k=0}^{m-1} \hT_m(f) (\frac{a+k}{m}, m(c-1))\\
&=& e^{- 2 \pi i a} \frac{1}{m} \sum_{k=0}^{m-1} \frac{1}{m}\Big(\sum_{j=0}^{m-1} f( \frac{\frac{a+k}{m} + j}{m}, m(m(c-1)))\Big) \\
&=& e^{-2 \pi i a}\frac{1}{m^2} \sum_{\ell=0}^{m^2-1} f( \frac{a+ \ell}{m^2}, m^2 c -m^2)).\\
\end{eqnarray*}
We now apply the twisted-periodicity relations to obtain
\begin{eqnarray*}
\hT_m^{\vee} ( \hT_m(f))(a, c) &=& e^{-2 \pi i a} \frac{1}{m^2} \sum_{\ell=0}^{m^2 -1} e^{2 \pi i m^2( \frac{a+ \ell}{m^2})} f(\frac{a+ \ell}{m^2}, m^2 c)\\
&=& e^{- 2 \pi i a} \frac{1}{m^2} \sum_{\ell=0}^{m^2 -1} e^{2 \pi i a} f( \frac{a+ \ell}{m^2}, m^2 c)\\
&=& \hT_{m^2}(f)(a, c),
\end{eqnarray*}
as asserted.
Now we already know that $\hT_m \circ \hT_{m} = \hT_{m^2}$. 
Since the operator $\hT_m$ is invertible on $\sP^{\ast}(\RR \times \RR)$, we conclude that
$\hT_m^{\vee} = \hT_m$ on this domain.

We obtain $\hS_m^{\vee}= \hS_m$ by an analogous calculation; details are omitted.  
\end{proof}

%
%
%

\subsection{$L^p$-spaces: Proof of Theorem \ref{th2001}} \label{sec44}

Analogous results on the two-variable Hecke operators $\hT_m$ apply in the
Banach spaces $L^p (\Box, da\,  dc)$ for $p \ge 1$. To study them
we note that functions on the extended function space $\sP^{\ast}( \RR\times \RR)$ are
determined by their values in the open unit square $\Box^{\circ}$. 
We let 
$$
\mathcal{D}_p :=\sP^{\ast}( \RR\times \RR)|_{\Box^{\circ}}\bigcap L^p( \Box, da \,dc).
$$
This domain is dense in $L^p ( \Box, da dc)$ because it contains $C^{0}(\Box)$,
viewed as those continuous functions on $\Box^{\circ}$ that continuously extend to $\Box$,
which is known to be dense (\cite[Theorem 3.14]{Rud74}).

\begin{lemma}\label{le49} 
For all functions $F(a, c)$ in  the domain  $\sD_p \subset L^p(\Box, da \,dc)$, the $L^p$-norms of  
  $\hT_m(F)$ and $\hS_m(F)$ restricted to $\Box^{\circ}$, for  each $m \ge 1$, satisfy 
$$ || \hT_m(F) ||_p \le m ||F||_{p} \qquad and \qquad 
|| \hS_m(F)||_p \le m ||F||_{p}  
$$
for all $p \ge 1$. 
\end{lemma}

\begin{proof}
The  $L^p$-norm  for $p < \infty$ is 
$$
||F||_p:= \Big( \int_{0}^1\int_{0}^1 |F(a, c)|^p da \,dc\Big)^{\frac{1}{p}},
$$
and the norm $||F||_\infty$ for $p=\infty$ {is the largest value among the suprema of $|F(a, c)|$ 
on the regions in $\Box^{\circ}$ where it is continuous.}
For $p \ge 1$ this norm $|| \cdot ||_p$ satisfies the triangle inequality
(by Minkowski's inequality).

{ Write}
\begin{equation}\label{eq491}
(\hT_m F)(a, c)= \sum_{j=0}^{m-1} (\hT_mF)_j(a, c),
\end{equation}
with 
$$
(\hT_mF)_j(a, c) : = 
\left\{
\begin{array}{cl}
\hT_m(F) (a, c)  & \mbox{if}~ \frac{j}{m} \le c< \frac{j+1}{m}~, \\
~~~ \\
0 & \mbox{otherwise}.
\end{array}
\right.
$$
Using twisted-periodicity we find for $\frac{j}{m} \le c< \frac{j+1}{m}$
that
$$
(\hT_m F)_j(a, c) = \frac{1}{m} \sum_{k=0}^{m-1} F(\frac{a+k}{m}, mc -j) e^{- 2 \pi i j (\frac{a+k}{m})}.
$$
{ Now, for $1 \le p < \infty$, letting
$$
\tilde{F}(\frac{a+k}{m}, mc -j) : = 
\left\{
\begin{array}{cl}
F(\frac{a+k}{m}, mc -j)  & \mbox{if}~ \frac{j}{m} \le c< \frac{j+1}{m}~, \\
~~~ \\
0 & \mbox{otherwise}, 
\end{array}
\right.
$$
we have}
\begin{eqnarray*}
|| \tilde{F}( \frac{a+k}{m}, mc -j)e^{-2 \pi i j(\frac{a+k}{m})} ||_p^p &=& 
 \int_{\frac{j}{m}}^{\frac{j+1}{m}} \Big( \int_{0}^1| F(\frac{a+k}{m}, mc- j)|^p da \Big) \, dc \\
&=& \int_{0}^1 \int_{\frac{k}{m}}^{\frac{k+1}{m}} | F(\tilde{a}, \tilde{c})|^p d \tilde{a} \, d \tilde{c}\\
& \le & \int_{0}^1 \int_{0}^1 | F(\tilde{a}, \tilde{c})|^p d \tilde{a} \, d \tilde{c}= ||F||_p^p.
\end{eqnarray*}
For $p=\infty$, a similar argument shows
$$|| \tilde{F}( \frac{a+k}{m}, mc -j)e^{-2 \pi i j(\frac{a+k}{m})} ||_\infty \le ||F||_\infty.$$
The $L^p$-triangle inequality then gives  $||(\hT_m F)_j||_p \le \frac{1}{m}  ( m ||F||_p ) = ||F||_p.$
Now \eqref{eq491} yields 
$$
||\hT_m(F)||_p \le \sum_{j=0}^{m-1} ||(\hT_m F)_j||_p \le m ||F||_p 
$$
for all $p \ge 1$, as required.

Interchanging the roles of $a$ and $c$ gives the same norm bound for
$$\hS_m(F)(a,c) =\frac{1}{m} \sum_{k=0}^{m-1} 
e^{2 \pi i ka} F \left( ma, \frac{c+k}{m} \right).$$ 
\end{proof}

\begin{proof}[Proof of Theorem \ref{th2001}.]
(1) Earlier lemmas show the four families of  operators  $\hT_m,$  $\hS_m,$  $\hT_m^{\vee},$ $\hS_m^{\vee}$
preserve the extended twisted-periodic space $\sP^{\ast}( \RR\times \RR)$.
The norm bound in Lemma \ref{le49} 
implies that they  also preserve each domain $\sD_p$, for $1 \le p \le \infty$.
Lemmas \ref{le46},  \ref{le47} and \ref{le48} 
together show that on each $\sD_p$ there are really only two distinct families of operators,
which are $\hT_m$ and $\hS_m$, and that  $\hS_m= \frac{1}{m}(\hT_{m})^{-1}$ holds there.  
{ It follows that } the four families of operators $\hT_m, \hS_m, \hT_m^{\vee}, \hS_m^{\vee}$ induce well-defined 
bounded linear actions on $L^p(\Box, da \, dc)$ for $1 \le p \le \infty$. Next for $F, G \in \sD_p$ we have $\hS_m= \frac{1}{m} \hT_{m}^{-1}$ and Lemma \ref{le49} gives
$$
|| F - G||_p \le m ||\hS_m( \hT_m(F)) - \hS_m(\hT_m(G))||_p \le m^2 ||\hT_m(F) - \hT_m(G)||_p.
$$
Since $\sD_p$ is dense, taking limits gives the same norm bounds for $F, G \in L^p(\Box, da\, dc)$. This
lower bound  inequality implies
that $\hT_m$ is invertible, and that 
 the relation  $\hS_m= \frac{1}{m} \hT_{m}^{-1}$ continues to hold on the extended operators on $L^p(\Box, da \, dc).$

The argument above also establishes the following upper and lower bounds for the norm of $\hT_m(F)$:
\begin{equation}
 \frac{1}{m^2}||F||_p \le   || \hT_m(F) ||_p \le m ||F||_{p}.
 \end{equation}
{  We improve this bound for the case $p=2$ below.}

(2) The pairwise commutativity of members of all  four sets of these operators on $L^p(\Box, da \, dc)$ extends by continuity
from their pairwise commutativity on the domain $\sD_p$, which they preserve by (1), { which itself follows from
the commutativity result established  in
 Lemma \ref{le47}.}
 
 (3) On $L^2( \Box, da \, dc)$ it  suffices to show 
 that the adjoint Hecke operator $(\hT_m)^{\ast} = \hS_m$, since applying $\ast$ then gives
 the other relation, {whence $\sA_0^2$ is a $\star$-algebra.}
 We need only check it holds on the dense subspace $\sD_2$ of $L^2( \Box, da \, dc)$.
 We have
 \begin{eqnarray*}
 \langle \hT_m(f)(a, c) , g(a, c) \rangle &=&
  \frac{1}{m} \sum_{k=0}^{m-1} \int_{0}^1 \int_{0}^1 f(\frac{a+k}{m}, mc) \overline{ g(a, c) } da \, dc\\
  &=& \frac{1}{m} \sum_{j=0}^{m-1} \sum_{k=0}^{m-1} \int_{\frac{j}{m}}^{ \frac{j+1}{m}}  \int_{0}^1
  f(\frac{a+k}{m}, mc- j) e^{- 2 \pi i j (\frac{a+k}{m})} \overline{ g(a, c) } da \, dc.
  \end{eqnarray*}
  Now we make the variable change $\tilde{c} = mc -j$ with $0 \le \tilde{c} < 1$ and
  $\tilde{a} = \frac{a+k}{m}$ with $ \frac{k}{m} \le \tilde{a} < \frac{k+1}{m}$. Since
  $a = m \tilde{a} -k$ and $da \, dc = d \tilde{a} d \tilde{c}$, we  obtain
  \begin{equation*}
 \langle \hT_m(f)(a, c) , g(a, c) \rangle = \frac{1}{m} \sum_{j=0}^{m-1} \sum_{k=0}^{m-1} \int_{0}^1
 \Big( \int_{\frac{k}{m}}^{\frac{k+1}{m}} f( \tilde{a} , \tilde{c}) e^{- 2 \pi i j \tilde{a}} 
 \overline{g  (m \tilde{a} - k, \frac{\tilde{c} + j}{m})} d \tilde{a} \Big) d \tilde{c}.
\end{equation*}
On the other hand
  \begin{eqnarray*}
 \langle f(a, c) , \hS_m(g)(a, c) \rangle &=& 
 \frac{1}{m} \sum_{j=0}^{m-1} \int_{0}^1 \int_{0}^1 f(\tilde{a}, \tilde{c})\,  \overline{ e^{2 \pi i j \tilde{a}} g(m\tilde{a},\frac{\tilde{c}+j}{m}) } \, d\tilde{a} \, d\tilde{c}\\
 &=& \frac{1}{m} \sum_{j=0}^{m-1} \sum_{k=0}^{m-1}  \int_{0}^1 \int_{\frac{k}{m}}^{\frac{k+1}{m}} 
 f(\tilde{a}, \tilde{c}) e^{- 2 \pi i j \tilde{a}} \overline{g( m \tilde{a}-k, \frac{\tilde{c} + j}{m})} d \tilde{a}  d\tilde{c}.
 \end{eqnarray*}
 Term-by-term comparison of the formulas shows $(\hT_m)^{\ast} = \hS_m$.
 
 { Now that we know the adjoint $(\hT_m)^{\ast} = \hS_m$,  we have
\begin{eqnarray} \label{unitary}
 || \hT_m(F)||_2^2  &= & \langle \hT_m(F), \hT_m(F)\rangle_2 = \langle (\hT_m)^{*} \circ \hT_m(F),F \rangle_2 \nonumber \\
 & = &   \langle \hS_m \circ \hT_m(F),F \rangle_2 =\langle \frac{1}{m} \hT_m^{-1} \circ \hT_m(F), F \rangle_2 \nonumber \\
 &= & \frac{1}{m} ||F||_2^2.
\end{eqnarray}
This yields the norm identity 
$|| \hT_m(F)||_2 = \frac{1}{\sqrt m} ||F||_2$ valid for all functions $F$. It follows that  $\sqrt{m} \hT_m$ is
a Hilbert space isometry, and since it is surjective it is a unitary operator. 
In similar fashion $\sqrt{m} \hS_m= \frac{1}{\sqrt{m}} \hT_m^{-1}$ is 
a surjective isometry, hence a unitary operator.}
\end{proof}

%
%

\section{Lerch Eigenspaces $\sE_s$ and their Properties}\label{sec5}
\setcounter{equation}{0}

We construct a two-dimensional space $\sE_s$ of simultaneous
eigenfunctions for all the two-variable Hecke operators, built out
of the Lerch zeta function. We make use of the functional equations
for the Lerch zeta function.

%
%
%

\subsection{Lerch zeta function: functional equations}\label{sec51}

The functional equations proved in part I \cite[Theorem 2.1]{LL1}
take a particularly simple form when expressed in terms of the 
 $\hR$-operator. Recall that for $0< a, c <1$ these functional equations involve the two
 functions
 $$
 L^{+}(s, a, c) = \sum_{n = -\infty}^{\infty} e^{2 \pi i n a} |n+c|^{-s}
 $$
 and
 $$
 L^{-}(s, a, c) = \sum_{n=-\infty}^{\infty} (\sgn (n+c)) e^{2 \pi i n a}|n+c|^{-s}.
 $$
 Now let $\Gamma_{\RR}^{+}(s) = \pi^{-\frac{s}{2}} \Gamma(\frac{s}{2})$ and 
 ${\Gamma_{\RR}^{-}(s)} = \pi^{\frac{s+1}{2}}\Gamma(\frac{s+1}{2}) =\Gamma_{\RR}^{+}(s+1)$.
 The functional equations are
 $$
 \Gamma_{\RR}^{+}(s)L^{+}(s, a, c) = e^{-2 \pi i a c}\Gamma_{\RR}^{+}(1-s) L^{+}(1-s, 1-c, a)
 $$
 and
 $$
 \Gamma_{\RR}^{-}(s)L^{-}(s, a, c) = i \, e^{-2 \pi i a c}\Gamma_{\RR}^{-}(1-s) L^{-}(1-s, 1-c, a).
 $$

%

\begin{defi}~\label{de 21} 
{\em 
The   {\em Tate  gamma functions}  $\ggG^{\pm}(s)$ 
(also known as {\em Gelfand-Graev gamma functions}, cf.  \cite[Chap. 2, Sect. 2.5]{GGP69}) are given by
\begin{eqnarray}
\ggG^{+}(s)  & := & \frac{ \Gamma_{\RR}^{+}(s)}{\Gamma_{\RR}^{+}(1-s)} = \frac{ \pi^{-\frac{s}{2}} \Gamma( \frac{s}{2})}
{ \pi^{-\frac{1-s}{2}} \Gamma(\frac{1-s}{2})} ,  \label{N223a} \\
\ggG^{-}(s)  & := & \frac{\Gamma_{\RR}^{-}(s)}{\Gamma_{\RR}^{-}(1-s)}= \frac{ \pi^{-\frac{s+1}{2}} \Gamma( \frac{s+1}{2})}
{ \pi^{-\frac{2-s}{2}} \Gamma(\frac{2-s}{2})}. \label{N223b} 
\end{eqnarray}
}
\end{defi}

These functions  satisfy the identities 
\beql{N224}
\ggG^{\pm}(s) \ggG^{\pm}(1-s) = 1 ~~~\mbox{for}~~ s \in \CC.
\eeq
The functions $\ggG^{\pm}(1-s)= \ggG^{\pm}(s)^{-1}$  have a ``scattering matrix''
 interpretation (see  Burnol \cite[Sec. 5]{Bur03}), in which \eqn{N224} expresses 
 unitarity of the ``scattering matrix''.
The functions  $\ggG^{\pm}(s)$ have simple zeros at certain positive integers and simple poles at
certain negative integers, whose location depends  on the sign $\pm$. \\

Recall also  that part I (\cite[Theorem 2.2]{LL1}) 
defined the {\em extended Lerch functions} $L_{*}^{\pm}(s, a, c)$ on the domain 
$(a, c) \in \RR \times \RR$;  these satisfy the twisted-periodicity conditions
\begin{eqnarray}
L_{*}^{\pm}(s, a+1, c)  &= &L_{\ast}^{\pm} (s, a, c), \label{N226a} \\
L_{*}^{\pm}(s, a, c+1) &=  &e^{  -2\pi i a} L_{\ast}^{\pm} (s, a, c). \label{N226b}
\end{eqnarray} 
These functions are entire functions of $s$ except for $L^{+}(s, a, c)$ when
$a$ or $c$ is an integer, which is meromorphic with its only singularities being
simple poles at $s=0$ or $s=1$, or both.

The $\hR$-operator acts on the Lerch functions as
\beql{N225}
\hR(L^{\pm})(s, a,c) = e^{-2 \pi i  a c} L^{\pm}(s, 1-c, a).
\eeq
The functional equations for the Lerch zeta function given in part I have the following
formulation  in terms of the  $\hR$-operator. \medskip

%
\begin{prop}~\label{pr21}
{\rm (Lerch Functional Equations)}
The functions  $L^{\pm}(s, a,c)$ satisfy for $0 < a, c <1$ the functional equations
\beql{N227}
\Gamma_{\RR}^{\pm}(s)L^{\pm}(s, a, c) = w_{\pm} \Gamma_{\RR}^{\pm} (1-s) e^{-2 \pi i a c} L^{\pm}(1-s, 1-c, a),
\eeq
in which $w_{+}=1, w_{-}=i$. 
Equivalently,
\beql{N228}
L^{\pm}(s, a, c) = w_{\pm} \gamma^{\pm}( 1-s) \hR(L^{\pm})(1-s, a, c).
\eeq
The same functional equations hold for the extended Lerch
functions $L_{\ast}^{\pm}(s,a,c)$ valid for 
 $(a, c) \in \RR \times \RR$.
\end{prop}

\begin{proof}
 For $0 < a, c < 1$ this result is shown  in part I \cite[Theorem 2.1]{LL1},
 {  where we have restated that result using the $\hR$-operator.}
   The extension to the boundary cases then uses the definition of $\zeta_{\ast}(s, a,c)$ given
in part I. It  applies on the domain $(a, c) \in \RR \times \RR$, given in \cite[Theorem 2.2]{LL1},
including integer values of $a$ and $c$.
\end{proof}

%
%
%

\subsection{ Lerch eigenspaces $\sE_s$}\label{sec52}

We  will view the complex variable $s \in \CC$ as fixed, and introduce the abbreviated notations
\begin{eqnarray} 
L_s^{\pm}(a, c) &  :=  & L^{\pm}( s, a, c) \label{200b} \\
R_s^{\pm}(a, c) & :=    & e^{-2 \pi i a c} L^{\pm}(1-s, 1-c, a). \label{200c}
\end{eqnarray}
The definition of the $\hR$-operator now gives
\beql{200g}
R_s^{\pm}(a,c)=  \hR(L_{1-s}^{\pm})(a,c).
\eeq
The   functional equations \eqn{N228}  given in Proposition \ref{pr21} 
above can be rewritten
\beql{N401}
L_s^{\pm}( a, c) = w_{\pm} \ggG^{\pm}( 1-s) R_s^{\pm}(a, c). 
\eeq
With this convention,  the $s$-variable in $R_{s}^{\pm}(a, c)$ corresponds to $1-s$ in
the associated Lerch function in \eqn{200g}, and this leads to 
$L_s^{\pm}(a, c)$ and $R_s^{\pm}(a,c)$ being linearly related in \eqn{N401}. \smallskip

\begin{defi}~\label{de31}
{\em
For fixed $s \in \CC$, the  {\em Lerch eigenspace} $\sE_s$ is 
the vector space over $\CC$ spanned by the four functions 
\beql{201aa}
\sE_s := <L_s^{+}(a, c), L_s^{-}(a, c), R_s^{+}(a, c), R_s^{-}(a, c)>
\eeq
for $(a, c) \in \Box^{\circ}= \{ (a, c): 0 < a, c < 1\}$.
}
\end{defi}

The vector space   $\sE_s$ { is} at most two-dimensional, since, as noted above, 
$L_s^{+}(a,c)$ and $R_s^{+}(a,c)$ (resp. $L_s^{-}(a,c)$ and $R_s^{-}(a,c)$)
are linearly dependent, with dependency relations given by the functional equation \eqn{N401}.

\begin{lemma}\label{dimEs}
For each $s \in \CC$, the  Lerch eigenspace $\sE_s$ is a two-dimensional space all of
whose members are real-analytic in $(a, c)$ on the domain $\Box^{\circ}$.
\end{lemma}

\begin{proof}
To see that $\sE_s$  is exactly two-dimensional
for all $s \in \CC$, 
note for  fixed $\Re(s) >0$ one has $\sE_s= <L_s^{\pm}(a,c)>$, with both functions of $(a, c)$ being
 linearly independent.
For  fixed $\Re(s) < 1$,  one has $\sE_s = <R_s^{\pm}(a, c)>$,
with both functions of $(a, c)$ being linearly independent. 
\end{proof}

 The reason for introducing four functions in \eqn{201aa} is that at integer values of $s$ at least one of the
four functions $L_s^{\pm}(a, c), R_s^{\pm}(a, c)$ is identically zero,
corresponding to poles in the term $\gamma^{\pm}(1-s)$ in \eqn{N228}.

One can easily derive two alternate expressions for Lerch eigenspaces, valid  for all $s \in \CC$, 
which  remove the effect of the gamma factors. They are
\beql{202k}
\sE_{s} = < L^{+}( s, a, c), L^{-}( s, a, c), e^{-2 \pi i ac}L^{+}(1-s, 1-c, a), e^{-2 \pi i ac}L^{-}(1-s, 1-c, a)>
\eeq
and 
\beql{202h}
\sE_{s} = < \zeta_{\ast}(s, a, c), \hR(\zeta_{\ast})(1- s, a, c),  \hR^2(\zeta_{\ast})(s, a, c),
 \hR^3(\zeta_{\ast})(1-s, a, c)>.
\eeq
{ It is easy to check, using the functional equation, that each of the functions on the right side is
in $\sE_s$. It remains to verify that at each integer $s$ one has two linearly independent elements
in the $(a, c)$-variables.}

%
%

\subsection{Operator properties of Lerch eigenspaces $\sE_s$}\label{sec53}

We  show that the functions in the Lerch eigenspace are
simultaneous eigenfunctions of the two-variable Hecke operators,
the Lerch differential operator $\hD_L$ and the 
involution operator $\hJ= \hR^2$. We also show the vector spaces $\sE_s$
are permuted by related operators.

%
%

\begin{theorem}\label{th201} {\rm (Operators on Lerch eigenspaces)}
For each $s \in \CC$, the functions in the two-dimensional vector space $\sE_s$ have the following properties.
\begin{itemize}
\item[(i)] 
{\rm (Lerch differential operator eigenfunctions)}
Each $f \in \sE_s$ is an eigenfunction of the Lerch differential operator 
$\hD_L = \frac{1}{2 \pi i} \frac{\pt}{\pt a}
\frac{\pt}{\pt c} + c \frac{\pt}{\pt c}$ with
eigenvalue $-s$, i.e.,
\beql{201a}
\hD_L f (s, a, c) = -sf (s, a, c) 
\eeq
holds at all $(a, c) \in \RR \times \RR$, with both $a$ and $c$ non-integers.
\item[(ii)] 
{\em (Simultaneous Hecke operator eigenfunctions)}
Each $f \in \sE_s$ is a simultaneous eigenfunction of all 
two-variable Hecke operators
$\{ \hT_m : m \ge 1 \}$ with eigenvalue $m^{-s}$, in the sense that
\beql{202a}
\hT_m f = m^{-s} f 
\eeq
holds on the domain $(\RR \smallsetminus \frac{1}{m}\ZZ)
\times (\RR \smallsetminus \ZZ)$.
\item[(iii)] 
{\rm ($\hJ$-operator eigenfunctions)} The space $\sE_s$ admits 
the involution  $\hJ f (a,c) = e^{-2 \pi ia} f(1-a, 1-c)$,
under which it decomposes into one-dimensional eigenspaces 
$\sE_s = \sE_s^+ \oplus \sE_s^-$ with eigenvalues $\pm 1$, namely
$\sE_s^{\pm} = < F_{s}^{\pm}>$ and 
\beql{203a}
\hJ (F_s^{\pm}) = \pm F_s^{\pm}.
\eeq
\item[(iv)]  
{\rm ($\hR$-operator action)} 
The $\hR$-operator acts by $\hR(\sE_s)= \sE_{1-s}$ with
\beql{118}
\hR(L_s^{\pm} ) = w_{\pm}^{-1} \ggG^{\pm}(1-s) L_{1-s}^{\pm},
\eeq
where $w_{+}= 1$ and $w_{-}=i$.
\end{itemize}
\end{theorem}

This result  yields the following consequence.

%
%

\begin{coro}\label{coro201} {\rm (Invariance of $\sE_s$ under  Hecke operator families)}
For each $s \in \CC$ the Lerch eigenspace $\sE_s$ is invariant under all four
families of real-analytic two-variable Hecke operators
$\{ \hT_m: m \ge 1\},$  $\{ \hT_m^{\vee}: m \ge 1\}$, $\{\hS_m: m \ge 1\},$
and $\{\hS_m^{\vee}: m \ge 1\}.$
\end{coro}

\begin{proof} 
These four sets of operators are $\{ \hR^j \hT_m \hR^{-j} \}$ for $j=0,1 ,2, 3$. The case $j=0$ is covered by (ii). 
By  (iv) the effect of $\hR$ is to map the four generating functions for $\sE_s$
to a permutation of them spanning $\sE_{1-s}$. Since the $\hR$
operator is  applied an even number of times for $j=1, 2, 3$, the final image is always $\sE_s$.
\end{proof}

To prove Theorem \ref{th201} we  first 
 determine the action of $\DLp$ and $\DLm$ on the eigenspace $\sE_s$.

\begin{lemma}\label{le302} {\rm (Raising and Lowering Actions)}

\begin{itemize}
\item[(i)]
For each $s \in \CC$, the operator 
$\DLm = \frac{1}{2 \pi i} \frac{\pt}{\pt a} +c$ has 
\beql{310}
(\DLm L_s^{\pm} )(s, a,c) ) = L_{s-1}^{\mp} (s, a,c ),
\eeq
so that it takes $\DLm ( \sE_s ) = \sE_{s-1}$.
\item[(ii)]
For each $s \in \CC$ the operator 
$\DLp = \frac{\pt}{\pt c}$ has
\beql{311}
(\DLp L_s^{\pm}) (s, a,c) = - s L_{s+1}^{\mp}(s, a,c),
\eeq
so that it takes  $\DLp(\sE_s ) = \sE_{s+1}$.
\end{itemize}
\end{lemma}

\begin{proof}
We establish (i) and (ii) for $\Re (s) > 1$ using the basis $L_s^{\pm}(a,c)$
for $\sE_s$ and the
Dirichlet series representations
\beql{312}
\zeta (s,a,c) = \sum_{n=0}^\infty e^{2 \pi ina} (n+c)^{-s} 
\eeq
and
\beql{312b}
e^{-2 \pi i a} \zeta(s, 1-a, 1-c) = \sum_{n=1}^{\infty} e^{-2\pi i n a}(n-c)^{-s}.
\eeq
Calculations can be done term-by-term on these
Dirichlet series  to verify \eqn{310}
and \eqn{311}.  For real $0 < a < 1$ and $0 < c < 1$ 
the Dirichlet series \eqn{312} converges conditionally for
$\Re (s) > 0$ and differentiating term-by-term can still be justified.
(Alternatively, we can analytically continue in the $s$-variable.)

For $\Re (s) < 0$ (resp. $\Re(s) <1$) we apply similar reasoning to the basis
 $R_s^{\pm}(a,c)$ for $\sE_s,$ which is possible since these functions
have absolutely (resp. conditionally)  convergent Dirichlet series expansion in this range. 
\end{proof}


\paragraph{\em Proof of Theorem~\ref{th201}.}
(i) Since $D_L = \DLm \DLp$, by Lemma~\ref{le302} the relation
\beql{319}
\hD_L L_s^{\pm} = -s L_{s}^{\pm}
\eeq
follows for all $s \in \CC$.
Since $\sE_s = <L_s^{+}(a, c) , L_s^{-}(a, c) >$ for $\Re (s) > 0$, this
proves (i) for $\Re (s) > 0.$ 
In the case $ \Re(s) < 1$
we use the alternate basis
$$\sE_s = <~e^{- 2 \pi i a c}L_{1-s}^+ (1- c,a)~,~ 
e^{-2 \pi i ac}L_{1-s}^{-}(1-c, a)>,
$$
see \eqref{202k}.

Since
$$e^{-2 \pi ia c} L_{1-s}^{\pm} (1-c,a) = \hR ( L_{1-s}^{\pm} (a,c)), $$
we may apply Lemma \ref{le301} (iii) and \eqn{319} to get
\begin{eqnarray*}
\hD_L (\hR (L_{1-s}^{\pm} (a,c))) & = & - \hR \hD_L (L_{1-s}^{\pm} (a,c)) 
- \hR L_{1-s}^{\pm} (a,c) \\
& = & - \hR (- (1 - s ) L_{1-s}^{\pm} (a,c)) - \hR L_{1-s}^{\pm} (a,c) \\
& = & - s \hR ( L_{1-s}^{\pm} (a,c)) ~.
\end{eqnarray*}
Thus  (i) follows in this case.\smallskip

(ii)  It suffices to verify that $\hT_m f = m^{-s} f$ on a basis of $\sE_s$.
For $\Re (s) > 0$ we use the basis
$\sE_s = < \zeta (s,a,c) , e^{-2 \pi ia} \zeta (s, 1-a, 1-c) >$.
Using the Dirichlet series expansion \eqn{312} we have, formally,
\begin{eqnarray*}
\hT_m \zeta (s,a,c) & = & \frac{1}{m} \sum_{k=0}^{m-1} 
\zeta \left( s, \frac{a+k}{m} , mc \right) \\
& = &
\frac{1}{m} \sum_{k=0}^{m-1} \sum_{n=0}^\infty 
e^{2 \pi i n ( \frac{a+k}{m} )} (n+mc)^{-s} \\
& = & \sum_{n=0}^\infty e^{2 \pi i \frac{na}{m}} \left( \frac{1}{m}
\sum_{k=0}^{m-1} e^{2 \pi i \frac{kn}{m}} \right) (n+mc)^{-s} ~.
\end{eqnarray*}
The inner sum is 0 unless $n \equiv 0$ $(\bmod~m)$, 
in which case it is 1, so that, setting $n=ml$, we obtain
$$
\hT_m \zeta (s,a,c) = \sum_{l=0}^\infty e^{2 \pi ila} (m(l+c))^{-s} = 
m^{-s}\zeta (s,a,c) ~.
$$
These operations are easily justified when $\Re (s) > 1$ and the sum 
converges absolutely.
They also hold for $0 < \Re (s) \le 1$ when the sums converge conditionally,
by using finite sums
$$\sum_{n=0}^{Nm} e^{2 \pi i n ( \frac{a+k}{m})} (n+mc)^{-s}$$
and then letting $N \to \infty$.
A similar argument shows that
$$
\hT_m (e^{-2 \pi ia} \zeta (s,1-a, 1-c)) = 
m^{-s} e^{-2 \pi ia} \zeta (s,1-a, 1-c)$$
when $\Re (s) > 0$, and this  completes the case $\Re (s) > 0$. For
the case $\Re (s) < 1$ { we work with a different basis of $\sE_s$: }
$$\sE_s = < e^{-2 \pi iac} \zeta (1-s, 1-c, a) ,~
e^{-2 \pi iac + 2 \pi ic} \zeta (1-s, c,1-a) >~.
$$
{ The proof is actually simpler. We shall show details for $e^{-2 \pi iac} \zeta (1-s, 1-c, a) : = F(a,c)$, 
the other basis function being similar. First assume that $\Re (s) < 0$ so that $\Re (1-s) > 1$ 
and $F(a,c)$ is given by an absolutely convergent series 
$$ F(a, c) = e^{-2 \pi iac} \sum_{n=0}^{\infty} e^{2 \pi i n(1-c)} (n + a)^{-(1-s)}.$$
Then 
\begin{eqnarray*}
\hT_m F(a,c) & = & \frac{1}{m} \sum_{k=0}^{m-1} e^{-2 \pi i(a+k)c} \sum_{n=0}^\infty e^{2 \pi i n(1-mc)}(n + \frac{a+k}{m})^{-(1-s)} \\
 & = & \frac{1}{m} m^{1-s} e^{-2 \pi iac} \sum_{n=0}^\infty \sum_{k=0}^{m-1} (mn +k +a)^{-(1-s)} e^{-2 \pi i c(mn+k)} \\
& = & m^{-s} F(a,c),
\end{eqnarray*}
as desired. For $0 \le \Re (s) < 1$ the series for $F(a,c)$ converges conditionally and we use the partial sum $\sum_{n=0}^N$ and let $N \to \infty$ to approach $F(a,c)$ as in the case $0 < \Re (s) \le 1$. This proves (ii)} \\  

(iii) Using $\hJ f(a,c) =e^{-2 \pi ia} f(1-a, 1-c)$ 
and the definition \eqn{200b} of $L_s^{\pm}$ it is easy to see that
\beql{320}
\hJ( L_s^{\pm}) = \pm L_s^{\pm} 
\eeq
holds for all $s \in \CC$.
This shows $\hJ : \sE_s \to \sE_s$ for all 
$s \in \CC \smallsetminus \{0, -1,-2, \ldots, \}$, where
$\sE_s = <F_s^+, F_s^{-}>$, and gives a splitting
$\sE_s = \sE_s^+ \oplus \sE_s^-$.
For the remaining values, we use the basis $R_s^\pm(a,c)$ for
$\sE_s$, and
a straightforward computation shows that
\beql{321}
\hJ (R_s^\pm) = \pm R_s^\pm
\eeq
holds for
all $s \in \CC$.
This shows that $\hJ : \sE_s \to \sE_s$ for 
$s \in \CC \smallsetminus \{ 1,2,3 \ldots \}$ where
$\sE_s = < R_s^+, R_s^- >$, and it gives a splitting 
$\sE_s = \sE_s^+ \oplus \sE_s^-$.
The functions $L_s^+$ and $R_s^+$ (resp. $L_s^-$ and $R_s^-$) 
are proportional for each $s \in \CC$
by the functional equation.\smallskip

(iv) The relation
\beql{313}
w_\pm \hR (L_s^{\pm})(a,c) =
{\ggG^\pm (s)} L_{1-s}^{\pm} (a,c) 
\eeq
 with $w_+= 1$ and $w_- = i$ 
 follows from the symmetrized functional equation (\ref{N228})
by replacing $s$ with $1-s$.
Since $\hR^4 = \hI$, the image of the vector space $\sE_s$ is 
necessarily a vector space
$\hR ( \sE_s )$ of the same dimension. We know that ${\rm dim} (\sE_s) =2$ by Lemma \ref{dimEs}.\medskip

If $s \in \CC \smallsetminus \ZZ$, then the coefficient on the 
right side of \eqn{313} is in $\CC \smallsetminus \{0\}$
for both choices of sign $\pm$ ,
which exhibits an explicit isomorphism $\hR : \sE_s \to \sE_{1-s}$.
If $s=s_0 \in \ZZ$, then exactly one of $\Gamma_{\RR}^\pm (s)$ and 
$\Gamma_{\RR}^\pm (1-s)$ has a pole at $s_0$, and
\eqn{313} is not well-defined.
What happens is that one of $L_s^{\pm}(a,c)$ or $L_{1-s}^{\pm} (a,c)$
is identically zero.
This can be seen by studying the functional equation \eqn{313} 
as $s$ varies, approaching the value $s_0$.
One side of the equation is defined and finite at $s_0$;
the other side must, for each fixed $(a,c) \in (0,1) \times (0,1)$ 
have the corresponding
$F_s^{\pm}(a,c) \to 0$ as $s \to s_0$, since it is an entire function of $s$.
However we can still see that 
$\hR (\sE_{s_0} ) = \sE_{1-s_0}$, by noting that
$\hR (\sE_s ) \subseteq \sE_{1-s}$ for $s \not\in \ZZ$, and letting
$s \to s_0$ we obtain
$\hR (\sE_{s_0} ) \subseteq \sE_{1-s_0} $, by analytic continuation in $s$.
The image is necessarily two-dimensional, since $\hR$ is invertible,
so  $\hR (\sE_{s_0} ) = \sE_{1-s_0}$ holds for $s=s_0$. 
$~~~\Box$ 

%
%

\subsection{Analytic properties of Lerch eigenspaces $\sE_s$}\label{sec54}

We establish the following  analytic properties of members of $\sE_s$,
which we deduce from results  in part I.

\begin{theorem}\label{th31}
{\rm (Analytic Properties of Lerch eigenspaces)}
For fixed $s \in \CC$ the functions in the Lerch eigenspace $\sE_s$ are real analytic
functions of $(a, c)$ on  $(\RR \smallsetminus \ZZ) \times (\RR \smallsetminus \ZZ)$,
which may be discontinuous at values $a, c \in \ZZ$. They 
have the following properties.
\begin{itemize}
\item[(i)] {\rm (Twisted-Periodicity Property)} All functions  $F(a,c)$ in $\sE_{s}$  satisfy the 
twisted-periodic functional equations
\begin{eqnarray} 
F(a+1,~c~) & =& ~~~~~~~~ F(a,c), \label{118c} \\
F(~a~, c+1)& = & e^{-2\pi i  a} F(a, c). \label{118d}
\end{eqnarray}

\item[(ii)] {\rm (Integrability Properties)}

(a) If $\Re(s) >0$,  then 
for each noninteger $c$ all functions  in $\sE_{s}$
have 
$f_c(a) :=F(a, c) \in L^{1}[(0,1), da],$ and  
all  their Fourier coefficients
\beql{118e}
f_n(c) := \int_{0}^1 F(a,c) e^{-2\pi i n a} da, ~~~~~~~n \in \ZZ,
\eeq
are  continuous
functions of $c$ on $0<c<1$.

(b) If $\Re(s)<1$,  then for each noninteger $a$ all functions  in $\sE_{{s}}$
have $ g_a(c):=e^{2 \pi i a c}F(a, c) \in L^{1}[(0,1), dc],$
and all Fourier coefficients
\beql{118f}
g_n(a) := \int_{0}^1 e^{2\pi i a c}F(a,c) e^{-2\pi i n c} dc, ~~~~~~~n \in \ZZ,
\eeq
 are  continuous
functions of $a$ on $0< a<1$.

(c) If $0 < \Re(s) < 1$ then all functions in $\sE_{s}$ belong to 
$L^{1}[\Box , da dc]$.
\end{itemize}
\end{theorem}

\begin{proof}
These properties will be deduced from results in part I.

 (i)  Theorem 2.2 of part I \cite{LL1} established
the twisted-periodicity functional equations for $\zeta_{\ast}(s, a, c)$. It follows
by repeated applications of $\hR$ that these functional equations also
hold for
$ e^{-2\pi i a c} \zeta(1-s, 1-c , a)$, 
$e^{-2\pi i a} \zeta(s, 1-a, 1-c)$, $e^{-2\pi i a( c+1)} \zeta(1-s, c, 1-a).$
These four functions span the two-dimensional vector space
$\sE_{s}$ for every $s \in \CC$.\\

(ii) Part I \cite[Theorem 6.1]{LL1} shows
for $s \in \CC \smallsetminus \ZZ$, that  subtracting off suitable members of the four basis functions
\begin{equation}
c^{-s},  e^{-2\pi i a}(1-c)^{-s}, 
e^{-2\pi i (1-a) c} (1-a)^{s-1}, e^{-2\pi i a c} a^{s-1}
\end{equation}
from the two functions $L^{\pm}(s, a, c)$ yields  functions
$\tilde{L}^{\pm}( s,a,c)$  that are continuous on the
closed unit square $\Box= [0,1]\times[0,1]$, and which therefore belong
to $L^2[ \Box, da dc].$ Since 
\[
\zeta_{\ast}(s, a, c)= L^{+}( s, a, c) + L^{-}( s, a, c),
\]
it also has a continuous extension to the closed unit square after
subtracting off suitable multiples of these
four functions. In fact only three of the four
basis functions are needed in the subtraction, for $\zeta_{\ast}(s, a, c)$  the function
$e^{-2\pi i a}(1-c)^{-s}$ is omitted, see \cite[Theorems 5.1 and  5.2]{LL1}.
At the  integer values of $s$ excluded, some of the terms
subtracted off have poles. \smallskip

We now consider the effect of the singularities of
these basis functions on the four sides of the unit square on determining
absolute integrability of $\zeta_{\ast}(s, a,c)$
 on horizontal and vertical lines in the unit square. 
In the $a$-direction,  for all $s \in \CC$  the two functions 
$c^{-s}$ and $e^{-2\pi i a}(1-c)^{-s}$ each  lie in $L^{1}[(0,1), da]$ for each fixed
value $0< c<1$, while  for $\Re(s) >0$  the two  functions
$e^{-2\pi i a c} (1-a)^{s-1}$ and $e^{-2\pi i a c + 2 \pi i c} a^{s-1}$ lie in
$L^{1}[(0,1), da]$ for each fixed
value $0< c<1$.  In  the $c$-direction, for $\Re(s) <1$ the two functions
$c^{-s},  e^{-2\pi i a}(1-c)^{-s}$ lie in $L^{1}[(0,1), dc]$ for each fixed
value $0< a<1$, while for all $s \in \CC$  the other two functions lie in 
$L^{1}[(0,1), dc]$ for each fixed
value $0< a<1$.  These properties are inherited by $\zeta_{\ast}(s, a, c)$
which establishes the $L^1$-membership part of (ii-a) and (ii-b) for $\zeta_{\ast}(s, a, c)$
when $s \in \CC \smallsetminus \ZZ$. Now we apply the $\hR$ operator,
interchanging $s$ and $1-s$ as necessary, and
deduce that (ii-a) and (ii-b) also hold for the other three functions
in \eqn{202h}. Since these functions span $\sE_s$, the $L^1$-membership properties
hold for all functions in $\sE_s$.  For the Fourier coefficient assertion,
the formula \eqn{101c} gives  a convergent Fourier series 
in the $a$-variable for $\zeta_{\ast}(s, a, c)$ valid for $\Re(s)>0$, and the
continuity of the Fourier coefficients 
$f_n(c) = (n+c)^{-s}$ for $n \ge 0$, $f_n(c) = 0$ for $n < 0$ is manifest.
Similarly we obtain continuity of the Fourier coefficients in the $a$-variable
of $e^{-2\pi i a}\zeta_{\ast}(s, 1-a, 1-c)$ when $\Re(s) >0$. These two functions
span $\sE_{s}$ for $\Re(s) >0$, which establishes (ii-a) for $s \in \CC \smallsetminus \ZZ$.
The Fourier coefficient
assertion (ii-b) is obtained by similar calculations. Multiplication
by $e^{2\pi i a c}$ does not affect membership in $L^1[(0,1), dc]$. 
 One finds for $F(a,c)= \hR(\zeta_{\ast})(1-s, a, c) \in \sE_{s}$
that $e^{2\pi i a c} F(a, c)= \zeta_{\ast}(1-s, 1-c, a)$ has Fourier coefficients 
in the $c$-variable 
$g_n(a) = |-n+a|^{s-1}$ for $n \ge 0$ and $g_n(a) = 0$ for $n <0$,
while $F(a,c)= \hR^3(\zeta_{\ast})(1-s, a, c) \in \sE_{s}$ has
$e^{2\pi i a c} F(a, c)= e^{2 \pi i c}\zeta_{\ast}(1-s, c, 1-a)$, and has Fourier 
coefficients in $c$-variable $g_n(a) = 0$ for $n \ge0$ and
 $g_n(a) = |-n+a|^{s-1}$ for $n <0$. These two functions span  $\sE_{s}$
 for $\Re(s) <1$, which establishes (ii-b) for $s \in \CC \smallsetminus \ZZ.$\smallskip

We next establish the remaining
cases of properties (ii-a) and (ii-b), which are (ii-a) for integer $s \ge 1$ and  (ii-b) for
integer $s \le -1$. For (ii-a) we directly use the Fourier series expansion
\eqn{101c}. After removing the term $c^{-s}$, which is clearly in $L^1[(0,1), da]$,
the remaining Fourier series for fixed $0< c< 1$ is absolutely integrable 
at all integers $s=n \ge 2$, and its Fourier coefficients
are continuous in $c$ by inspection.  A similar property holds for $\zeta_{\ast}(s, 1-a, 1-c)$,
after removing the term $e^{-2\pi i a}(1-c)^{-s}$. Since for $\Re(s)>0$ these functions span
$\sE_s$, property (ii-a) holds in these cases. For
the remaining case  $s=1$ Rohrlich showed (see Milnor \cite[Lemma 4]{Mi83}) that both the 
functions $\zeta_{\ast}(1, a, c) - c^{-1} = \sum_{n=1}^{\infty} \frac{e^{2\pi i na}}{n+c}$
and  $\zeta_{\ast}(1, 1-a, 1-c) - \frac{e^{2\pi i a}}{(1-c)}$
are  in $L^1[(0,1), da]$ with the given Fourier coefficients, completing this case.
The corresponding property (ii-b) for $s=n \le-1$ is established similarly.\smallskip

Property  (ii-c) is shown in  part I \cite[Theorem 2.4]{LL1}.
\end{proof}

%
%

\section{Hecke Eigenfunction Characterization of  Lerch Eigenspaces $\sE_s$}\label{sec6}
\setcounter{equation}{0}

In this section we characterize the Lerch eigenspace
as being the complete set of simultaneous eigenfunctions of the
family of two-variable Hecke operators $\{ \hT_m: m \ge 1\}$ that satisfy
some auxiliary integrability and continuity conditions on the function. This can be viewed
as a generalizing Milnor's characterization of 
space of Hurwitz zeta functions, which we  first explain in \S6.1.
Our characterization theorem is given in \S6.2. We note that 
this characterization theorem does not impose any eigenfunction
condition with respect to the differential operator $D_L$.

%
%

\subsection{Milnor's theorem for Kubert functions}\label{sec61}

A  {\em Kubert operator} $\hT_m$ is an operator that acts formally as
\beql{901}
\hT_m(f)(x) := \frac{1}{m} \sum_{k=0}^{m-1} f \left( \frac{x+k}{m}\right).
\eeq
These operators were studied by Kubert and Lang \cite{KL76}, \cite{KL81}  and Kubert \cite{Ku79},
acting on a group, for example $\RR /\ZZ$.
In 1983 Milnor \cite{Mi83} characterized simultaneous eigenfunction solutions
to the Kubert operator, acting on continuous functions on the open interval $(0,1)$.

Milnor  studied
the family of operators $\hT_m : \rC^0 ((0,1)) \to \rC^0 ((0,1))$ given by \eqn{901}.
These operators form a commuting family of operators on $\rC^0 ((0,1))$.

\begin{theorem}\label{th41}
{\rm (Milnor)} 
Let $\sK_s$ denote the set of continuous functions $f: (0,1) \to \CC$ which
satisfy
\beql{202}
\hT_m f(x) = m^{-s} f(x)  \quad \mbox{for all} \quad
x \in (0,1) ~,
\eeq
for each $m \ge 1$.
Then $\sK_s$ is a two-dimensional
complex vector space and consists of real-analytic
functions. Furthermore
$\sK_s$ is an invariant subspace for the involution
$$\hJ_0 f(x) := f(1-x)$$ 
and decomposes into one-dimensional eigenspaces 
$\sK_s =  \sK_s^+ \oplus  \sK_s^-$ 
which are spanned by  an even
eigenfunction $f_s^+ (x)$ and an odd eigenfunction $f_s^- (x)$, respectively,
 which satisfy
\beql{203}
\hJ_0 f_s^{\pm} (x) = \pm f_s^{\pm}(x) ~.
\eeq
\end{theorem}

\begin{proof}
This is proved in \cite[Theorem 1]{Mi83}.
\end{proof}

Milnor gave an explicit  basis for $\sK_s$, which for
$s\neq 0, -1, -2, \ldots$ is given in terms of the 
(analytic continuation in $s$ of the) Hurwitz zeta function
\beql{204}
\zeta_s (x) := \zeta (s,0,x) = \sum_{n=0}^\infty \frac{1}{(n+x)^{s}},
\eeq
namely   
\beql{205}
\sK_s = < \zeta_{1-s} (x), \zeta_{1-s} (1-x) >.
\eeq 
Properties (ii) and (iii) of Theorem \ref{th201} are 
analogous to those in Theorem \ref{th41}, in which the
variable $x$ in the Kubert operator \eqref{901}
is identified with the variable $a$,  and the second variable $c$
is set to $0$.

Milnor observes that 
$\frac{\pt}{\pt x}$ maps $\sK_s$ to $\sK_{s-1}$, acting as
a ``lowering operator''.
Because the individual  operators inside the sum on the right side of \eqn{901}
are contracting, this  ``lowering operator'' suffices
in his proof.\smallskip

In \S3  we observed,  in the two-variable context, that
$\frac{1}{2 \pi i} \frac{\pt}{\pt a} + c : \sE_s \to \sE_{s-1}$ is 
a ``lowering'' operator
while $\frac{\pt}{\pt c} : \sE_s \to \sE_{s+1}$ is a ``raising'' operator.
Property (i) of Theorem \ref{th201} is derived using these properties.
Milnor's theorem {\em formally} corresponds to setting $a=x$ and $c=0$ in 
Theorem \ref{th201},
except that $c=0$ falls outside the domain of definition of the functions 
we consider.

%
%

\subsection{Characterization of  Lerch eigenspaces $\sE_s$}\label{sec62}

Milnor's proof of Theorem~\ref{th41}
used in an essential way  the property  that for   ``Kubert operators'' $\hT_m$ all 
terms on  the right side of
\eqn{901} are contracting operators on the domain $x \in (0,1)$.
In contrast, the  two-variable
Hecke operators  \eqref{102} are expanding  in the $c$-direction.
To deal with the  expanding property we  impose extra analytic conditions on 
the function in the whole plane
$\RR \times \RR$, in order to obtain a characterization   of $\sE_s$ as
being simultaneous eigenfunctions of two-variable Hecke operators.

Our main result shows that 
the twisted-periodicity and integrabilities  properties 
of Theorem~\ref{th31} yield such a characterization.\smallskip


\begin{theorem}~\label{th62}
{\rm (Lerch Eigenspace Characterization)}
Let $s\in \CC$. Suppose that
$F(a, c): (\RR \smallsetminus \ZZ) \times (\RR \smallsetminus \ZZ) \to \CC$ 
is a continuous function  that satisfies the following conditions.
\begin{itemize}
\item[(1)] {\em (Twisted-Periodicity Condition) }
For $(a,c) \in (\RR \smallsetminus \ZZ) \times (\RR \smallsetminus \ZZ)$, 
\begin{eqnarray}
F(a+1, ~c~) & =& ~~~~~~~~F(a,c),   \label{421a}\\
F(~a~, c+1) &=& e^{-2\pi i a} F(a,c) \label{421b}.
\end{eqnarray}

\item[(2)] {\em (Integrability Condition) } At least one of the following two
conditions (2-a) or (2-c) holds.

{\rm (2-a)} The $s$-variable has
$\Re(s) > 0$. 
For   $0<c<1$ each function $f_c(a):= F(a, c) \in L^1[(0,1), da]$, and    
  all  the Fourier coefficients 
 \[
   f_n(c) := \int_{0}^{1} f_c(a)e^{-2\pi i n a} da= \int_{0}^1 F(a,c) e^{-2\pi i n a}da,~~~n \in \ZZ,
   \]
  are continuous functions of $c$.

 {\rm (2-c)}  The $s$-variable has
 $\Re(s) < 1$. For  $0<a<1$ each function $g_a(c):= e^{2 \pi i a c}F(a,c) \in L^1[(0,1), dc]$,  and
 all  the Fourier coefficients 
 \[
   g_n(a) := \int_{0}^{1} g_a(c)e^{-2\pi i n c} dc= \int_{0}^1 e^{2\pi i a c}F(a,c) e^{-2\pi i n c}dc,~~~n \in \ZZ,
   \]
 are continuous functions of $a$.

  \item[(3)] {\em (Hecke Eigenfunction Condition)} For all $m \ge 1$, 
\beql{422a}
\hT_{m}(F)(a, c) = m^{{-s}}F(a, c)
\eeq
holds on the domain 
$\{ (a, c) \in (\RR \smallsetminus \ZZ)\times (\RR \smallsetminus \frac{1}{m}\ZZ)\}$.
     \end{itemize} 
    
\noindent Then  
$F(a, c)$ is the restriction to noninteger $(a, c)$-values
of a function in the Lerch eigenspace $\sE_{s}$.
\end{theorem}

\paragraph{\bf Remarks.} (i)  Theorem~\ref{th31} shows that
all functions in  $\sE_s$  for $\Re(s) >0$  satisfy conditions (1), (2-a) and (3),
and all functions  in  $\sE_s$  for $\Re(s) < 1$  satisfy conditions (1), (2-c) and (3) above.
{ Conditions  (2-a), (2-c) between them cover all $s \in \CC$,
and they  hold simultaneously inside the critical strip $0 < \Re(s) <1$.}

 (ii) The function $F(a, c) := c^{-s}$ satisfies  properties (2-a)  and  (2-c) and also
the eigenvalue property (3). However it  fails to satisfy the twisted-periodicity property (1).  

\begin{proof}
We first treat the case when  condition (2-a) holds,
where $F(a,c)$ is absolutely integrable on horizontal lines in the
unit square.
The proof uses the Fourier series expansion of $F(a, c)$ with respect
to the $a$-variable, which we write as
\beql{423a}
F(a,c) \sim \sum_{n = -\infty}^{\infty}  f_{n}(c)e^{2\pi i n a}, 
\eeq
where $\sim$ means that  no assertion is made about convergence of
the Fourier series to $F(a,c)$.

Condition (2-a) shows that all the Fourier coefficients
\beql{423b}
f_n(c) := \int_{0}^{1} F(a,c) e^{-2\pi i n a} da
\eeq
are continuous functions of $c$ for $0< c<1$. The Fourier coefficients
are also defined for $\ell < c < \ell+1$ for all integer $\ell$, using
the twisted-periodicity property \eqn{421b} in the $c$-variable,
$F(a, c+\ell)= e^{-2\pi i \ell a} F(a, c)$, and they satisfy
\beql{440}
f_n(c+\ell) = f_{n+\ell}(c) ~~\mbox{for~all}~~\ell\in \ZZ
\eeq
by uniqueness of the Fourier series expansion { for $L^1$-functions.}

The action of the two-variable Hecke operator 
\[
\hT_m(F)(a, c) : =\frac{1}{m} \sum_{k=0}^{m-1} F(\frac{a+k}{m}, mc)
\]
is well-defined pointwise for
 $(a,c) \in (\RR \smallsetminus \ZZ)\times(\RR \smallsetminus \frac{1}{m}\ZZ)$.
The resulting function is in $L^1[(0,m), da]$ for $c \in \RR \smallsetminus \frac{1}{m}\ZZ$, and
is periodic of period {$1$  in the $a$-variable because 
$$ \hT_m(F)(a+1, c) = \hT_m(F)(a, c) + \frac{1}{m}\left( F(\frac{a+m}{m}, mc) - F(\frac{a}{m}, mc)\right) = \hT_m(F)(a, c)$$ by (\ref{421a}). Its Fourier expansion on $L^2[(0,1), da]$ is }
\begin{eqnarray}\label{612}
\hT_m(F)(a, c) & = &
 \frac{1}{m} \sum_{k=0}^{m-1} \sum_{n \in \ZZ} f_n(mc) e^{-2\pi i n (\frac{a+k}{m})} \nonumber \\
&=& 
\sum_{n \in \ZZ} \left( \frac{1}{m} \sum_{k=0}^{m-1} e^{-2\pi i n (\frac{a+k}{m})}\right) f_n(mc)\nonumber \\
&=& \sum_ {\ell \in \ZZ} f_{m\ell}(mc) e^{-2 \pi i \ell a}. \label{425a}
\end{eqnarray}

Now suppose that $\hT_m(F)(a, c)= m^{-s} F(a,c)$ on the indicated domain. 
Comparison of \eqn{425a} with the Fourier series expansion of $m^{-s} F(a,c)$
in \eqn{423a} gives
\beql{441}
f_{mn}(mc) = m^{-s} f_{n}(c)~~~\mbox{for}~~~ c \in \RR \smallsetminus \frac{1}{m}\ZZ.
\eeq
To simplify later formulas, we set
\beql{442}
\tilde{f}_n(c) := |n+c|^{s} f_n(c).
\eeq
Then \eqn{440} and \eqn{442} yield, for all $l \in \ZZ$, 
\beql{443}
\tilde{f}_n(c + \ell) = f_n(c+\ell)|n+c + \ell |^{s} = f_{n+\ell}(c)|n+c+\ell|^{s} = \tilde{f}_{n+\ell}(c).
\eeq
Furthermore \eqn{441} gives
\begin{eqnarray}
\tilde{f}_{mn}(mc) &=& f_{mn}(mc)|mn+mc|^{s}\nonumber \\
&=& (m^{-s} f_n(c))m^s |n+c|^{s} \nonumber \\
&=& \tilde{f}_n(c) .  \label{444}
\end{eqnarray}
We now determine all solutions to \eqn{443} and \eqn{444}.
Consider $n=0$, and we obtain, for all $m \ge 1$,
\beql{445}
\tilde{f}_0(mc)= \tilde{f}_0(c). 
\eeq
The right side of \eqn{445} is continuous for $c \in (0,1)$ which implies that the left side 
makes sense as a  continuous function 
for $mc \in (0, m)$.  Since $m$ is arbitarily large, we conclude that $\tilde{f}_0(c)$ 
extends to a continuous function on $(0, \infty)$.
Now \eqn{440} gives
\[ 
f_{0}(c-1) = f_{-1}(c) ~~\mbox{for}~~0< c< 1,
\]
and condition (2-a) gives the continuity of $f_{-1}(c)$ on this interval, so it
follows that  $\tilde{f}_0(c)$ is continuous
on $(-1,0)$.  As above \eqn{445} implies that $\tilde{f}_0(c)$ extends to a continuous
function on $(-\infty, 0)$.
Thus any possible discontinuity  of $\tilde{f}_0(c)$ is at $c=0$. 
Now {writing  $\tilde{c}= \frac{m_1}{m_2} c$ for positive integers $m_1, m_2$ and applying \eqn{445},} 
we  obtain for positive $\tilde{c}$ that
\[
\tilde{f}_0(\tilde{c})= \tilde{f}_0(m_2 \tilde{c})= \tilde{f}_0(m_1c)= \tilde{f}_0 (c)
\]
and similarly for negative $\tilde{c}$. Thus $\tilde{f}_0(c)= \tilde{f}_0(rc)$ for
all positive rational numbers $r$, which with the continuity results implies that
$\tilde{f}_0(c)$ is constant on $(-\infty, 0)$ and on $(0, \infty)$, say
$\tilde{f}_0(c) = A$ (resp. $B)$ on $(0, \infty)$ (resp. $(-\infty, 0)$).
Thus  we obtain
\[
f_0(c) = \left\{
\begin{array}{cl}
A|c|^{-s}  & \mbox{if} ~~c>0 \\
~~~ \\
B|c|^{-s}  & \mbox{if}~~ c<0~.
\end{array}
\right.
\]
Now \eqn{440} gives
\[
f_n(c) = f_0(c+n) =\left\{
\begin{array}{cl}
A|c+n|^{-s} & \mbox{if} ~~c>-n, \\
~~~ \\
B|c+n|^{-s} & \mbox{if}~~ c < -n~.
\end{array}
\right.
 \]
Thus the Fourier series of $F(a,c)$ agrees term-by-term with the Fourier series of 
\beql{447}
H(a,c): = \frac{1}{2}({A}+{B}) L^{+}(s, a, c) +
   \frac{1}{2}({A}-{B}) L^{-}(s, a, c).
\eeq
Since we have $\Re(s)>0$, this is in $L^1[(0,1), da]$ for noninteger $c$. 
So by uniqueness of Fourier series and continuity
we  conclude that $F(a,c)= H(a,c) \in \sE_s$ everywhere on
$(\RR \smallsetminus \ZZ) \times (\RR \smallsetminus \ZZ)$.

We now treat the case that condition (2-c) holds. The proof is similar in
spirit. We set
\[ 
G(a, c) := e^{2\pi i a c} F(a,c) ~~\mbox{for}~~(a,c) \in (\RR \smallsetminus \ZZ) \times (\RR \smallsetminus \ZZ).
\]
This function satisfies the modified twisted-periodicity conditions. 
\begin{eqnarray}
G(a+1, ~c~) & =& e^{2\pi i c}G(a,c)   \label{451a}\\
G(a, ~c+1) &=& ~~~~~~~~ G(a,c) \label{451b}.
\end{eqnarray}
Condition (2-c) guarantees that 
$G(a,c) \in L^1[(0,1), dc]$ for non-integer values of $a$,
so that it  has a Fourier expansion in
the $c$-variable: 
\[
G(a, c) \sim \sum_{n \in \ZZ}   g_n(a) e^{2\pi i n c},
\]
and condition (2-c) asserts that  the $g_n(a)$ are continuous functions of $a$,
for $0< a< 1$.  The twisted-periodicity condition \eqn{451a}
now implies that the Fourier coefficient functions 
$g_n(a)$ are defined for all  $a \in \RR \smallsetminus \ZZ$, and satisfy
\beql{453}
g_n(a + \ell) = g_{n -\ell}(a).
\eeq
By hypothesis
\[
e^{2\pi i a c} \hT_m(F)(a, c)  =   e^{2\pi i a c}(m^{-s} F(a, c)) = m^{-s}G(a, c)
\]
holds on $(\RR \smallsetminus \ZZ) \times (\RR \smallsetminus \frac{1}{m}\ZZ)$.
Thus we have the Fourier series
\beql{454}
e^{2\pi i a c} \hT_m(F)(a, c) \sim \sum_{n \in \ZZ}  m^{-s}g_n(a)e^{2 \pi i nc}.
\eeq
We evaluate the left side by expanding the Hecke operator, and note that
\[
G_k(a,c) :=     e^{2 \pi i a c} F(\frac{a+k}{m}, mc) = e^{-2\pi i kc}{G(\frac{a+k}{m}, mc)}
\] 
satisfies the twisted-periodicity conditions
\begin{eqnarray}
G_k(a+m, ~c~) & =& e^{2\pi i mc}G_{k}(a,c)   \label{455a}\\
G_k(a, ~c+1) &=& ~~~~~~~~ G_{k}(a,c) \label{455b}.
\end{eqnarray}
Condition (2-c) allows us to deduce that this function is in
$L^1[(0,1), dc]$, so it has a Fourier series expansion in the $c$-variable, which is
\[
G_k(a,c) \sim e^{-2 \pi i kc} \sum_{n \in \ZZ} g_n(\frac{a+k}{m}) e^{2\pi i n(mc)}.
\]
Summing up over $0 \le k \le m-1$ we obtain 
\begin{eqnarray*}
e^{2 \pi i  a c} \hT_m(F)(a,c)  \sim 
\frac{1}{m} \sum_{k=0}^{m-1}
\left(\sum_{n \in \ZZ} g_n(\frac{a+k}{m}) e^{2 \pi i (mn -k)c}\right).  
\end{eqnarray*}
By uniqueness of Fourier series of $L^1$-functions, we obtain 
  $m^{-s} g_{mn-k}(a) = \frac{1}{m}g_n(\frac{a+k}{m})$,
which we rewrite as
\beql{456}
g_{mn-k}(a) = m^{s-1} g_n(\frac{a+k}{m}),
\eeq
which is valid for $\frac{a+k}{m} \in \RR \smallsetminus \ZZ.$
We now set
\[
\tilde{g}_n(a) := g_n(a) |a-n|^{1-s},
\]
and note that
\[
\tilde{g}_{mn-k}(a)   =   g_{mn-k}(a) |a-mn+k|^{1-s} = m^{s-1}g_n(\frac{a+k}{m})|a-mn+k|^{1-s}
 = \tilde{g}_n(\frac{a+k}{m}).
\]
On choosing $n=k=0$ we obtain for all $m \ge 1$ that
\beql{457}
\tilde{g}_0(a) = \tilde{g}_0(\frac{a}{m})
\eeq
is valid for $\frac{a}{m} \in \RR \smallsetminus \ZZ.$
Now $g_0(a)$ is continuous on $(0,1)$, so the left side of this equation implies
that it extends to a continuous function on $(0, m)$ for all $m$, hence on $(0, \infty)$.
Now $g_0(a-1) = g_{1}(a)$ shows by condition (2-c) that $g_0(a)$ is continuous
on $(-1,0)$, and \eqn{457} for all $m$ implies that $\tilde{g}(a)$ continuously extends to
$(-\infty, 0).$ Since $\tilde{g}_0(\frac{a}{m_1})= \tilde{g}_0(m_2)$ for all positive
integers $m_1, m_2$, we conclude $\tilde{g}_0(a) = \tilde{g}_0(ra)$ for all positive rational
$r$, which with the continuity conditions forces $\tilde{g}_0(a)$ to be constant on
$(-\infty, 0)$ and on $(0, \infty)$, say $\tilde{g}_0(a) = A$ on $(0, \infty)$, resp.
$B$ on $(-\infty, 0)$. We deduce that
\[
g_0(a) = \left\{
\begin{array}{cl}
A|a|^{s-1}  & \mbox{if} ~~a>0, \\
~~~ \\
B|a|^{s-1}  & \mbox{if}~~ a<0~.
\end{array}
\right.
\]
Now \eqn{453} gives
\[
g_n(a) = g_0(a-n) =\left\{
\begin{array}{cl}
A|a-n|^{s-1} & \mbox{if} ~~a>n ,\\
~~~ \\
B|a-n|^{s-1} & \mbox{if}~~ a < n~.
\end{array}
\right.
 \]
Thus the Fourier series of $G(a,c)$ agrees term-by-term with the Fourier series of 
\beql{459} 
H(a,c): =\frac{1}{2} ({A}+{B}) e^{-2 \pi i ac}L^{+}( 1-s, 1-c, a) +
   \frac{1}{2}({A}-{B}) e^{-2\pi i a c}L^{-}(1-s, 1-c, a).
\eeq

\noindent Since $\Re(s)<1$  this function is in $L^1[(0,1), dc]$ for
non-integer $a$, and we may conclude by continuity that $G(a, c) = H(a,c) \in \sE_s$,
for $(a,c) \in ( \RR \smallsetminus \ZZ) \times(\RR \smallsetminus \ZZ).$
This completes the proof.
\end{proof}

%
%
\section{Concluding Remarks }\label{sec7}

This paper studied
two-variable Hecke operators  $\hT_m$ given by \eqref{102}
on spaces of functions of two  real variables. 
It  showed  that the Lerch zeta function $\zeta(s, a, c)$ is a simultaneous
eigenfunction of all $\hT_m$ with eigenvalue $m^{-s}$ for all $s \in \CC$. 
As mentioned in Section \ref{sec12}, 
we may formally define a {\em zeta operator} by 
 $$
 \hZ : = \sum_{m=1}^{\infty}\, {\hT}_m. 
 $$
 For a fixed complex value $\Re(s) >1$ we can make sense of
 this operator and observe that it has the Lerch zeta function as an eigenfunction and 
  the Riemann zeta value  $\zeta(s)$ as an eigenvalue; that is,
 $$
 \hZ(\zeta)(s, a, c) = \zeta(s) \zeta(s, a, c).
 $$
 In this paper we have extended the action of the individual operators $\hT_m$ to
arbitrary complex values of $s$ on suitable function spaces.
In particular one can define the individual $\hT_m$ in the Hilbert  space
$L^2( \Bx, da \, dc)$ when $0 < \Re(s)<1$. 
In this context, one may ask whether the other  structures attached to the Lerch
zeta function  in these four papers can 
yield insight into  the Riemann hypothesis. 

Concerning the two-variable Hecke operators themselves and the Riemann hypothesis
we  make the following observations.

\begin{itemize}
 \item[(1)]
 At $a=0$ the  operator $\hT_m$ degenerates  to the  dilation
  operator $\tilde{\hT}_m (f) (c) = f(mc)$. 
  In 1999 B\'{a}ez-Duarte \cite{BD99} noted that this  family of dilation operators relates to
  the real-variables approach to the Riemann hypothesis due to 
  Nyman  \cite{Nym50} and Beurling \cite{Beu55},  see also B\'{a}ez-Duarte \cite{BD03}, 
 Burnol \cite{Bur04}, \cite{Bur04a} and Bagchi \cite{Bag06}.  Therefore one may ask whether
 there is a Riemann hypothesis criterion directly formulable in terms of the
 two-variable Hecke operators $\hT_m$.
  
  \item[(2)]
  The  Lerch zeta function at $a=0$ reduces for $\Re(s)>1$ 
  to the Hurwitz zeta function. The Hurwitz zeta function  inherits the 
   discontinuities of the Lerch zeta function at integer values of $c$.
  Milnor \cite[p. 281]{Mi83} noted that
  at the  value $s=1$ the space $\sK_s$ includes on $(0,1)$ the odd function
  $ c - \frac{1}{2} $ which, due to  
   the discontinuities, extends to  the periodic function  $\beta_1(c):= \{ c- \frac{1}{2}\}$,
  the first Bernoulli polynomial,  a fractional part function.
 The fractional part function appears in the various  
real-variables forms of the Riemann hypothesis  above.
 The  discontinuities of the Lerch zeta function at integer values of $a$ or $c$
 (for some values of $s$)  
 represents an important feature of these functions, worthy of further
 study in this context.
\end{itemize}

An interesting connection to the Riemann hypothesis relates to  the
differential operator $\hD_L=\hD_L^{+} \hD_L^{-}$ considered here (and in \cite{LL2}, \cite{LL3})
for which the Lerch zeta function is an
eigenfunction. 

\begin{itemize}
\item[(1)]
 It is natural to consider the {\em symmetrized Lerch differential operator}
\beql{811}
\Delta_L := \frac{1}{2} \left( D_L^{+}D_L^{-} + D_L^{-} D_L^{+}\right) =
  \frac{1}{2 \pi i} \frac{\partial}{\partial a}\frac{\partial}{\partial c} + 
  c \frac{\partial}{\partial c} + \frac{1}{2} {\bf I}.
 \eeq
 This operator has many features of a Laplacian
 operator. It is formally skew-symmetric, and satisfies
$$
( \Delta_L\zeta)(s, a, c) = -(s- \frac{1}{2}) \zeta(s,a , c),
$$
so that the line of skew-symmetry is the critical line $Re(s) = \frac{1}{2}$.
This operator has the ``$xp$'' form suggested by Berry and Keating \cite{BK99}, \cite{BK99b}, 
 as the appropriate  form for  a ``Hilbert-Polya'' operator encoding the zeta zeros as
 eigenvalues.  
 
 \item[(2)]
 The operator $\Delta_L$ commutes with all the $\hT_m$ on
 the two-dimensional Lerch eigenspace $\sE_s$, which is however not contained in $L^2(\Bx, da \, dc)$.
 It  formally commutes with the $\hT_m$, but its commutativity depends on
 the specified domain of the unbounded operator $\Delta_L$, viewed inside $L^2(\Bx, da \, dc)$.
 Such a domain is specified in \cite[Sect. 9.2]{Lag15}, for which the resulting operator $\Delta_L$ has
 purely continuous spectrum.
 
 \item[(3)]
In order to view $\Delta_L$ as  a suitable Hilbert-Polya
operator for zeta zeros along these lines, it may be that one must instead  find a 
scattering on $\sH$ and a small closed subspace
of $\sH$ carrying the  operator $\Delta_L$. Related viewpoints on
Hilbert-Polya operators have been proposed by Connes \cite{Co99a}, \cite{Co99b},
 Burnol \cite{Bur01}  and the first author \cite{La06}.
 \end{itemize}

   \paragraph{\bf Acknowledgments.} 
    The authors thank Paul Federbush for helpful remarks regarding
Fourier expansions in Theorem~\ref{th62}. {  The authors thank the two reviewers
for many helpful comments and corrections. In particular we thank one of them for 
the observation \eqref{unitary} which strengthened Theorem 2.1 (3).}
 This project  was initiated   at
AT\&T Labs-Research when the first author worked there and the second author consulted there;
they thank AT\&T for support. The first author  received
support from the Mathematics Research Center at Stanford University
in 2009-2010. The second author received support from the National Center for
Theoretical Sciences and National Tsing Hua University in Taiwan in 2009-2014.
To these institutions the authors express their gratitude. 

%
%


\end{document}